\documentclass[11pt,openany,leqno]{article}
\usepackage{amsmath,amsthm,amsfonts,amssymb,amscd,url,color}
\usepackage{graphics}

\input{xy} \xyoption{all}
\usepackage[latin1]{inputenc}

\usepackage{enumerate}
\usepackage{hyperref}
\usepackage{epsfig} 
\input{xy} \xyoption{all}

\def\R{\mathbb R}
\def\I{\mathbb I}
\def\N{\mathbb N}
\def\Z{\mathbb Z}

\def\leb{\text{Leb}}

\def\arr{\overleftarrow}
\def \u {\underline }

\newtheorem{prop}{Proposition}[section]
\newtheorem{thm}[prop] {Theorem}

\newtheorem{conj}[prop] {Conjecture}
\newtheorem{defi}[prop] {Definition}
\newtheorem{lemm}[prop] {Lemma}

\newtheorem{cor}[prop]{Corollary}
\newtheorem{theo}{Theorem}

\newtheorem{Claim}[prop]{Claim}
\newtheorem{fact}[prop]{Fact}

\newtheorem{coro}[theo]{Corollary}

\newtheorem{conjecture}{Conjecture}
\newtheorem{problem}[prop]{Problem}
\theoremstyle{remark}
\newtheorem{exam}[prop]{Example}
\newtheorem{rema}[prop]{Remark}

\begin{document}

\title{Emergence
 and non-typicality of the finiteness of the attractors in many topologies}

\author{Pierre Berger\footnote{CNRS-LAGA, Uiversit\'e Paris 13, USPC.}}

\date{
In memoriam of  Anosov's 80th  anniversary.
}

\maketitle
\abstract{ 
We will introduce the notion of Emergence for a dynamical system, and we will conjecture the local typicality of super complex  ones. 
Then, as part of this program,  we will provide sufficient conditions for an open set of $C^{d}$-families of $C^r$-dynamics to contain a Baire generic set formed by families displaying infinitely many sinks at every parameter, for all $\infty \ge r\ge d\ge 1$ and $d<\infty$ and two different topologies on families. In particular the case $d=r=1$ is new. 
\section{Introduction}
\subsection{Tentatives to describe typical dynamics}
Under the dual leadership of Anosov-Sinai in USSR and Smale in the USA, the hyperbolic theory for differentiable dynamical systems grew up. 
We shall recall some elements of this theory.

Let $M$ be a manifold  and let $f$ be a $C^1$-diffeomorphisms of $M$. 
A compact set $K\subset M$ is hyperbolic if the tangent space of $TM| K$ is split into two vector sub-bundles $E^s$ and $E^u$, which are both $Df$-invariant and respectively contracted and expanded by the dynamics:
\[TM| K = E^s\oplus E^u\quad 
Df(E^s)=  E^s\quad Df(E^u)= E^u\quad \lim_{+\infty} \|Df^n|E^s\|= 0\quad \lim_{+\infty} \|Df^{-n}|E^u\|= 0\]

There are many examples of hyperbolic sets such as the Anosov maps (when $K=M$), the Smale Horseshoes, the Derivated of Anosov, and the Plykin attractors for diffeomorphisms, see \cite{Sm} for more details. 

An important property of the hyperbolic sets is their structural stability:
\begin{thm}[Anosov \cite{An67}]
For every $C^1$-perturbation $f'$ of $f$, there is a unique hyperbolic set $K'$ for $f'$, which is homeomorphic to $K$ via a map $h\colon K\to K'$ $C^0$-close to the canonical inclusion $K\hookrightarrow M$ and which conjugates the dynamics:
\[h\circ f|K= f'\circ h|K\; .\]
\end{thm}

The hyperbolic set $K$ is a \emph{basic set} if it is \emph{transitive} (there is a dense forward orbit in $K$) and \emph{locally maximal}: there is a neighborhood $N$ of $K$ such that $K= \cap_{n\in \Z} f^n(N)$. 

Smale defined the diffeomorphisms satisfying \emph{Axiom A} as those whose non-wandering set\footnote{The set of points $z\in M$ such  that any neighborhood $V$ of $z$ intersects one of its iterates: $\exists n\not= 0: f^n (V)\cap V\not= \varnothing$.} is the finite disjoint union of basic sets.

A basic set $K$ is a \emph{hyperbolic attractor} if $K= \cap_{n\ge 0} f^n(N)$.  Another  important result of this theory is:
\begin{thm}[Sinai-Bowen-Ruelle]
Given a hyperbolic attractor, there exists a unique invariant probability measure $\nu$ on $N$ so that for Lebesgue almost every $z\in N$, for every continuous function $\phi$:
\[\lim_{n\to \infty} \frac1n \sum_{i=0}^{n-1} \phi({f^i(z)})= \int \phi d\mu\; .\]
\end{thm}
This result was very appreciated by physicists since it enables one to get from a deterministic system a new system with robust statistical properties, and somehow to make a conceptual bridge between Classical Mechanics and Statistical mechanics. 

Also Smale made the following conjecture:
\begin{conj}[Smale 1965]
A Baire generic\footnote{A set is Baire generic if it contains  a countable intersection of open dense sets.} diffeomorphism of a compact manifold satisfies  
Axiom A. 
\end{conj}
This conjecture (and others by Smale and by Thom) 	appear at the beginning of a mathematical optimistic movement aiming to describe a typical dynamical system.

However, Smale and Smale-Abraham (1966) fund soon a counter example to this conjecture. 

In 1974,  a student of Smale, Newhouse discovered an extremely complicated new phenomenon, occurring in  a locally Baire generic set of dynamics.
\begin{thm}[\cite{Newhouse}]
 For every $r\ge 2$, for every manifold $M$ of dimension $\ge 2$, there exist a non-empty open set $U\subset Diff^r(M)$ and a generic set $\mathcal R\subset U$ so that for every $f\in \mathcal R$, the dynamics $f$ has infinitely many sinks, each of which having very different statistical properties. 
\end{thm}
Clearly these dynamics do not satisfy Axiom A (which have finitely many attractors). Even today, we do not know to describe a single example of these dynamics -- in the meaning that -- we do not know even if Lebesgue almost every point belongs to the basin of an \emph{ergodic attractor}. 

From  \cite{PS89}, an \emph{ergodic attractor}  $(\Lambda,\mu)$ is a compact transitive set $\Lambda$ supporting an invariant probability measure $\mu$ s.t. for a set of positive Lebesgue measure $B$ (called \emph{Basin}) it holds:
\[\lim_{n\to \infty} \frac1n \sum_{i=0}^{n-1} \phi({f^i(z)})= \int \phi \, d\mu,\quad \forall \phi\in C^0(M,\R)\quad \forall z\in B.\]

In the mean time the simulations of the atmospheric physicist Lorenz showed a new chaotic attractor for an ODE \cite{Lo63}.  
Later H\' enon modeled the first return of map of this flow, to get a simple paradigmatic example of a chaotic surface map: the H\'enon map $(x,y)\mapsto (x^2+a+y,-b x)$ parametrized by $a,b\in \R$. He conjectured that for $b=0.3$ and a certain parameter $a$ this map has a chaotic attractor \cite{He76}. During a series of papers, this conjecture have been shown to be true for $b$ sufficiently small.
\begin{thm}[Benedicks-Carleson  \cite{BC2}, Mora-Viana \cite{Mora-viana}, Benedicks-Young \cite{BY}, Wang-Young \cite{YW}, Takahasi \cite{T11}, Berger \cite{berhen},Yoccoz (1990-Today)]
For $b$ sufficiently small, for a set of Lebesgue positive measure of parameters $a$, the map $(x,y)\mapsto (x^2+a+y,-bx)$ has unique ergodic attractor $(\Lambda,\mu)$ , which is not supported by an attracting periodic orbit.
\end{thm}
The conjecture of H\' enon was a posteriori disturbing since an arbitrarily small neighborhood $N$ of the attractor is a topological disk, and so the attractor cannot be a hyperbolic attractor (otherwise there would be a line field on the topological disk $N$). Also the simplicity of the model, its physical meaning and the concept of abundance involved make this phenomenon unavoidable.

That is why the next conjectures have been formulated thanks to the concept of typicality sketched by Kolmogorov during his plenary talk 
in the ICM 1954. Here is a version of typicality which appears in many conjectures:
\begin{defi}[Arnold-Kolmogorov typicality] 
A property $\mathcal P$ on dynamics of a manifold $M$ is typical 
if there exists a Baire generic set of $C^d$-families $(f_a)_{a\in \R^k}$ of $C^r$-dynamics so that $\mathcal P$ is satisfied by Lebesgue almost every small parameter $a$. 
\end{defi}
Hence this definition of typicality involved integers $k,d,r$.  We will discuss about the topological spaces of families in the next section. 

To take into account the aforementioned examples and counter examples, there were several conjectures claiming the typicality of the finiteness of attractors, let us recall the following\footnote{Similar conjectures had been formulated by Tedeschini-Lalli \& Yorke, Palis \& Takens, and Palis him-self, some of them for low dimensional dynamical systems.}: 

\begin{conj}[Pugh-Shub \cite{PS95}]\label{ConjPS}
Typically (in the sens of Arnold-Kolmogorov) a diffeomorphism of a compact manifold has a finite numbers of topological attractors (and so sinks). 
\end{conj}


These conjectures aimed to model typical dynamics thanks to finitely many attractors. The general strategy though to prove them was to study the unfolding of stable and unstable manifolds (in analogy with Thom-Mather works in singularity Theory).

%

Recently, in \cite{BE15}, a mechanism has been found to stop the unfolding for an open set of dynamics' families. This mechanism is given by the \emph{parablender}, a generalisation of Bonatti-Diaz Blender for parameter families. This enabled to prove:  

\begin{thm}[\cite{BE15,BE152}]
For every manifold of dimension at least $3$, for every $k\ge 0$, for every $r> d\ge 1$, there exists an open set of $\hat{\mathcal U}$ of $C^d$-families $(f_a)_{a}$ of $C^r$-diffeomorphisms of $M$, so that for a generic $(f_a)_a\in \mathcal U$,  \emph{for every} parameter $a\in [-1,1]^k$, the map $f_a$ has infinitely many sinks.

\end{thm}
The same statement is also possible for surface local diffeomorphisms\footnote{
A local diffeomorphism is a differentiable map of a manifold so that  at every point its derivative is a linear bijection.}.

The \emph{main result} of this paper is devoted to get the case $r\ge d\ge 1$, $r\le \infty$ and $d<\infty$. Hence the cases $d=r$ or $r=1$ are new. It will be stated in section \ref{Statement of the main theorem}.

It will be proved by using a variation of the previous proof: we will put a source in the covered domain of the parablender. 
This enables a  revision of the previous proof, which is shorter and covers the new case $d=r\ge 1$. 

Also the statement of the main result lies on general hypothesis satisfied by an open set of families. This is motivated by a work in progress with S. Crovisier and E. Pujals, showing the Kolmogorov-Arnold $C^r$-typicality of dynamics displaying infinitely many sinks.

Such results are very disturbing since the general trend was to use the bifurcation theory to show the finiteness of attractors. Here the bifurcation theory enables to stop the bifurcation and shows the non-typicality of the finiteness of attractors.

\subsection{Emergence}

One of my personal motivations is the following problem:
 \begin{problem}\label{Problemstat}
 Show the existence of an open set of deterministic dynamical systems which typically cannot be described by means of statistics.
 \end{problem}
 
This problem goes in opposition to the aforementioned optimistic movement, as well as the massive (and naive) use of statistic in many branches of science (economy, ecology, physics ...). 

The aim is not to prove that statistics never apply (they do for many systems!), but that they do not apply for many typical systems, even among the finite dimensional, deterministic differentiable dynamical systems.  
We shall formalize this problem. For this end, we are going to define the Emergence of dynamical systems. This concept evaluates the complexity to approximate a system by statistics. 

In statistic it is standard to use the Wasserstein distance $W_1$ on the space of probability measures $\mathbb P(M)$ of a compact manifold $M$:
\[W_1(\nu,\mu) = \sup_{\phi\in Lip^1(M,[-1,1])}\int_M \phi(x)\, d(\mu-\nu)(x)\;,\quad \forall \nu,\mu\in \mathbb P(M)\]
where $Lip^1(M,[-1,1])$ is the space of $1$-Lipschitz functions with values in $[-1,1]$. 

Given a differentiable map $f$ of $M$, $x\in M$ and $n\ge 0$, we denote by $\delta_{n\, x}$  the probability measure which associates to an observable $\phi\in C^0(M,\R)$ the mean $\frac1n \sum_{k=0}^{n-1}\phi({f^k(x)})$.  

\begin{prop}
Given a probability  measure $\mu$, the following functions are continuous:
\[x\in M\mapsto d_{W^1}(\delta_{n\, x},\mu)\in \R\; ,\quad \forall n\]
\end{prop}
\begin{proof}
We notice that it suffices to show that for every $\delta>0$, there exists $\eta>0$ such that if $x$ and $x'$ are $\eta$ distant, then 
\[d_{W^1}(\delta_{n\, x'},\mu)\ge d_{W^1}(\delta_{n\, x},\mu)-\delta\; .\]
We recall that $Lip^1(M,[-1,1])$ endowed with $C^0$-uniform norm is compact, by Arzel\`a-Ascoli Theorem. Hence, there exists $\phi \in L^1(M,[-1,1])$ such that: $$d_{W^1}(
\delta_{n\, x},\mu)= \frac1n \sum_{k=0}^{n-1}\phi({f^k(x)})- \int_M \phi d\mu\; .$$ As $\phi$ and  $(f^k)_{k\le n}$ are Lipschitz, there exists $\eta>0$ so that for $x'$ $\eta$-close to $x$, it holds:
 \[ \frac1n \sum_{k=0}^{n-1}\phi({f^k(x')})\ge
 \frac1n \sum_{k=0}^{n-1}\phi({f^k(x)})-\delta\Rightarrow d_{W^1}(\delta_{n\, x'},\mu)\ge d_{W^1}(\delta_{n\, x},\mu)-\delta\; .\]
\end{proof}

We recall that the space of probabilities over a compact manifold and endowed with the metric $d_{W^1}$ is relatively compact.

Hence, given a differentiable map of $f$ of $M$, we can define the \emph{Emergence $\mathcal E(f,\epsilon)$ of $f$ at scale $\epsilon>0$} as the minimum numbers $N$ of probability measures $\{\mu_i\}_{1\le i\le N}$  so that 
\[\limsup_{n\to \infty} \int_{x\in M} \min_{1\le i\le N} d_{W^1}(\frac1n \sum_{k=0}^{n-1}\delta_{f^k(x)},\mu_i) \, d\leb\le \epsilon\; .\]

\begin{defi}[Emergence]
The Emergence is \emph{\bf F} if $\mathcal E(f,\epsilon) = O(1)$ when $\epsilon\to 0$.

The Emergence is at most \emph{\bf P}  if there exists $k> 1$ so that $\mathcal E(f,\epsilon) = O(\epsilon ^{-k})$.

The Emergence is \emph{\bf Sup-P} if $\limsup \frac{\log \mathcal E(f,\epsilon)}{-\log\epsilon}=+\infty$.
\end{defi}
We notice that the Emergence is a lower bound on 
the complexity (in space\footnote{The number of data to store.} and in time) to approximate numerically a  dynamical system by statistics with  precision  $\epsilon$. Following, the celebrated Cobham's thesis, an algorithm in Sup-P is -- in practical -- not feasible \cite{Co65}. 

Note that the Emergence is invariant by differentiable conjugacy. 

\medskip

\noindent{\bf Examples with {\bf F}-Emergence}
If a dynamical system $f$ admits finitely many  ergodic attractors $(\Lambda_i,\mu_i)_{1\le i\le N}$ whose basins $(B_i)_i$ cover Lebesgue almost all the manifold, then the Emergence is bounded by $N$ (and so it is of type { F})
\begin{proof} By the dominated function theorem, it suffices to show that for every $i\le N$ and every $x\in B_i$,  $d_{W^1}(
\delta_{n\, x},\mu_i)\to 0$. By compactness of $Lip^1(M,[-1,1])$, for every $n$, there exists $\phi_n \in Lip^1(M,[-1,1])$ so that: 
\[\Delta_{ n}:= d_{W^1}(\delta_{n\, x},\mu_i)= 
\int_M \phi_n d \delta_{n\, x}- \int_{M} \phi_n d\mu_i
= \frac1{n} \sum_{k=0}^{n-1}\phi(f^k(x))- \int_{M} \phi_n d\mu_i
 \; .\]
Let $\phi\in Lip^1(M,[-1,1])$ be a cluster value of $(\phi_n)_n$ and let $(n_j)_{j\ge 0}$   be an increasing sequence so that $\phi_{n_j}\to \phi$. Then 
\[\Delta_{ n_j}\le  2\|\phi_{n_j}-\phi\|_{C^0}+
 \frac1{n_j} \sum_{k=0}^{n_j-1}\phi(f^k(x))- \int_M 
\phi d\mu_i\to 0\; .\]
Thus every cluster value of $(\Delta_{ n})_n$ is zero, and so this sequence converges to zero.
\end{proof}
\begin{rema}
We recall that a diffeomorphism satisfying Axiom A,   an irrational rotation or a H\'enon map for Benedicks-Carleson parameters have finitely many ergodic attractors whose basin cover Lebesgue almost all the phase space $M$. Hence their Emergences are finite.
\end{rema}

\noindent{\bf Example with { P}-Emergence.}
Let $f$ be the identity. Observe that $\mathcal E(f,\epsilon) = O(\epsilon^{-n})$ with $n$ the dimension of $M$. Hence its Emergence is polynomial. Also the Emergence of an irrational rotation on a cylinder, which is the product of systems with Emergences 1 and $O(\epsilon^{-1})$, is $O(\epsilon^{-1})$.

It seems also possible to prove that the Emergence of the so-called Bowen eyes dynamics is $O(\epsilon^{-1})$. 

\medskip

Hence it seems that all the well understood dynamical systems have an Emergence at most {P}.
However, the main conjecture of this work states that those of Sup-P Emergence should not be neglected:
\begin{conjecture}\label{mainconj}
There exists an open set $U\subset Diff(M)$ so that a typical $f\in U$ has Emergence { Sup-P}.
\end{conjecture} 
Let us explain why a proof of this conjecture would solve Problem \ref{Problemstat} from the computational view point. Given a typical $f\in U$, to describe by means of statistics with precision $\epsilon$, all of its orbit, but a proportion Lebesgue measure $1-\epsilon$, we would need at least a super-polynomial number of invariant probabilities w.r.t. $\frac1\epsilon$. To find them by means of statistics, we need \emph{at least} one data for each of them, and so to do a super polynomial of number of operations. By Cobham's thesis this is not feasible by a computer.   

Also we notice that when the Emergence is Sup-P, the Hausdorff dimension of the set of probabilities which would model our system is infinite.

Hence to find these invariant probabilities, we would not be able to use the (finite dimensional) parametric statistics, but only the non-parametric ones, whose computational cost is higher (and much more than 1 as in the above lower bound).

Furthermore let us notice that even if we quotient the phase space by a symmetry group of finite dimension (as for the case of a rotation on the disk or the identity on a manifold), the Emergence of the system will remain Sup-P.

\medskip

Note that it is not even easy to find a locally typical non-conservative system with infinite Emergence (that is not in $F$), on the other hand KAM theory provides examples (of at least $P$-Emergence) in the conservative setting.

\medskip

\noindent{\bf Candidates for Sup-P-Emergence.} It is perhaps  possible to construct a unimodal map with Sup-P Emergence from \cite{HK90}, or a locally $C^r$-dense set of surface diffeomorphisms with Sup-P Emergence from \cite{Ki15}.  It would be very challenging to derivate from these systems one which is moreover locally typical. 

\begin{Claim} Dynamics with infinitely many sinks do not have finite Emergence.\end{Claim} 
\begin{proof} Indeed, for every $N\ge 1$, there exists $N$ different attracting cycles which define measures $\{\nu_i\}_{1\le i\le N+1}$. For $\epsilon>0$ small enough, the measure of each basin of $\nu_i$ is at least $\ge \sqrt \epsilon$. For a possibly smaller $\epsilon>0$, given any $N$-uplet of probability measure $\{\mu_i\}_{1\le i\le N}$, there exists exists $1\le j\le N+1$ so that $W_1(\nu_j, \{\mu_i\}_{1\le i\le N})\ge \sqrt \epsilon$. Consequently the Emergence $\mathcal E(\epsilon)$ at scale $\epsilon$ is greater than $N$. Hence the limit of $\mathcal E(\epsilon)$ as $\epsilon\to 0 $ is infinite.  \end{proof}

It is perhaps possible to make a variation of Newhouse's construction to produce a generic dynamics with { Sup-P} Emergence. This issue will be discussed in a forthcoming work.   
That is why main Theorem \ref{main} enters in this program.

Let me mention also the concept of universal dynamics of Bonatti-Diaz \cite{BD02} and Turaev \cite{Tu15} which might produce locally Baire generic sets of diffeomorphism with high Emergence.

\medskip 
It would be interesting to study Conjecture \ref{mainconj} w.r.t. different notions of typicality \cite{HK10} and smoothness.  Also it might be interesting to investigate the concept of Emergence for other metrics than $W_1$ on the space of invariant probability measures. 

\medskip 

Also it would be interesting to provide numerical evidences for such a program (from big data?). The following problem remains open.
\begin{problem}
 Show numerical simulations depicting a (typical) dynamical systems which displays infinitely many sinks. 
\end{problem}
Let us point out that by definition, a Sup-P Emergent  dynamical system is very complex to describe, and so the non-existence of such pictures is  consistent with their conjectured local typicality.
 }

\thanks{I am grateful to  V. Baladi, J. Bochi, C. Bonatti, F. Ledrappier, M. Lyubich, M. Shub, C. Tresser, D. Turaev and especially to  S. Crovisier and E. Pujals for many inspiring and motivating discussions. 
I thanks the anonymous referee for all his valuable suggestions and corrections.
\medskip 
\begin{center}
\emph{With all my thought to my master Jean-Christophe Yoccoz.}
\end{center}
}

\section{Statement of the main Theorem}

\subsection{Topological spaces of families of maps}
\paragraph{Space of families}
For $d\ge r\ge 0$ and $k\ge 0$, there are at least two ways to define a space of $C^d$-families parametrized by $\I^k:=[-1,1]^k$ of $C^r$-maps from a manifold $M$ into a manifold $N$. 

The first is attributed to Arnold by Y. Iliachenko (in the case $d=r$) \cite{IL99}. It is the space 
\[C^{d,r}_A(\I^k , M,N):= \{(f_a)_a: \partial_a^i\partial_z^j f_a (z)\text{ exists continuously } \forall i\le d\text{ , }i+j\le r \text{ and } (a,z)\in \I^k\times M\}\]
It has the advantage to be invariant by composition: for every  $(f_a)_a, (g_a)_a\in  C^{d,r}_A(\I^k , M,M)$, the composed family $(f_a\circ g_a )_a$ is in $C^{d,r}_A(\I^k , M,M)$.

Another way was presented in \cite{PS95} to  state Conjecture \ref{ConjPS}.  It is the space:
\[C^{d,r}_{PS}(\I^k , M,N):= \{(f_a)_a: \partial_a^i\partial_z^j f_a (z)\text{ exists continuously } \forall j\le r\text{ , }i\le d \text{ and } (a,z)\in \I^k\times M\}\]
It has the inconvenient \emph{to not be} invariant by  composition when $d>0$ and $r<\infty$. But it has the advantage to have a geometric meaning. A family is $(f_a)_a\in  C^{d,r}_{PS}(\I^k , M,M)$    is actually  a $C^d$-map from $\I^k $ into the Fr\'echet manifold $C^r(M,N)$. 

We remark:
\[C^{d,d+r}_{A}(\I^k, M,N) \subset C^{d,r}_{PS}(\I^k, M,N) \subset C^{d,r}_{A}(\I^k, M,N) \]

Hence in the important case $r=\infty$ (or $d=0$) the spaces $C^{d,r}_{PS}(\I^k, M,N)$ and $C^{d,r}_{A}(\I^k, M,N)$ are equal, and so they are denoted by $C^{d,\infty }(\I^k, M,N)$ (resp. $C^{0,r }(\I^k, M,N)$).

\paragraph{Topologies on families}
Any Riemanian metrics on $M$ and $N$, together with the Eulidean norm on $\R^k$ define a Riemannian metric on $N$, $TM^*\otimes TN$, ... ,
 $(\R^{*k})^{\otimes  i} \otimes (TM^*)^{\otimes  j} \otimes TN$. The topology of $C_A^{d,r}(\I^k ,M,N)$  is defined thanks to the following base of neighborhoods:
 \[V(f,K,\epsilon, r')= \{f'\in C_A^{d,r}(\I^k ,M,N) : 
d(\partial_a^i\partial_z^j f_a (z) ,\partial_a^i\partial_z^j f'_a (z) ) <\epsilon, \; \forall (a,z)\in K,\; i+j\le r', i \le d\}\]
among any finite  $r'\le r$, $\epsilon>0$ and any compact subset $K$  of $\I^k\times M$. The topology on $C_{PS}^{d,r}(\I^k ,M,N)$ is defined similarly ($i+j\le r'$ is replaced by $j\le r'$).  Both topologies coincide for $r=\infty$ and $d=0$.  

We remark that for $d=r$ the space  $C^{d,\infty }_A(\I^k, M,N)$ is canonically homeomorphic to the space $C^{d}(\I^k\times M,N)$ endowed with the $C^d$-compact-open topology. Also for $d=r=\infty$ the space  $C^{\infty,\infty }(\I^k, M,N)$ is canonically homeomorphic to the space $C^{\infty}(\I^k\times M,N)$ endowed with the compact-open, weak Whitney topology. A family in $C^{\infty,\infty }(\I^k, M,N)$ is called \emph{smooth}.

\subsection{Hyperbolic sets involved}

Most of the proofs involve surface local diffeomorphisms. We recall 
that a map $f\in C^r(M,M)$ is a local diffeomorhism if $r\ge 1$ and there is an open  covering $(U_i)_i$ of $M$ so that $f|U_i$ is a diffeomorphism onto its image for every $i$. 

Let us recall some elements of the hyperbolic theory for local diffeomorphisms. 

An invariant compact set $K$ for $f$ is \emph{hyperbolic} if there 
is a vector bundle $E^s \subset TM|K$ which is invariant by $Df|K$, contracted by $Df$ and so that the quotient $TM|K/E^s$ is expanded by the action induced by $Df$.

Then for every $z\in K$,  the following set, called the \emph{stable manifold} of $z$, is a  $\dim\, E^s$-manifold, injectively $C^r$-immersed  into $M$:
\[W^s(z; f):= \{z'\in M:  \; \lim_{+\infty}  d(f^n(z), f^n(z'))=0\}\]  

The notion of unstable manifold needs to consider the \emph{space of preorbits} $\arr K:=\{(z_i)_{i\le 0}\in K^{\Z^-}: z_{i+1}= f(z_i)\; \forall i<0\}$ of $K$. Given a preorbit $\arr z = (z_i)_{i\le -1}\in \arr K$, we can define the \emph{unstable manifold} of $\arr z$, which is a  $\text{codim}\, E^s$-manifold $C^r$-immersed  into $M$:
\[W^u(\arr z; f):= \{z'\in M:  \; \lim_{+\infty}  d(f^n(z), f^n(z'))=0\}\]  
In general this manifold is \emph{not} immersed \emph{injectively}.

When $z\in K$ is periodic,  the unstable manifold $W^u(z;f)$ denotes the one associated to the unique preorbit of $z$ which is periodic. 

A local stable manifold  $W^s_{loc} (z; f)$ of $z$ is an embedded, connected submanifold equal to a neighborhood of $z$ in $W^s(z; f)$. The local unstable manifold are defined similarly. We can chose them so that they depend continuously on $z$ and $\arr z$ respectively.

We endow $\arr K$ with the topology induced by the product topology of $K^\Z$. Hence  $\arr K$ is compact. Note that when $f|K$ is bijective, $\arr K$ is homeomorphic to $K$.

\paragraph{Blender} A hyperbolic set $K$ of a surface local diffeomorphisms $f$ is \emph{blender} if $\dim E^u\not=2$ and a continuous union of local unstable manifolds $\cup_{\arr  z\in \arr K} W^u_{loc} (\arr z; f)$ contains robustly a non empty open set $O$ of $M$:
\[\cup_{\arr z\in \arr K} W^u_{loc} (\arr z; f')\supset O \quad ,\quad \forall f' \; \text{$C^1$-close to } f\; .\]
The set $O$ is called a \emph{covered domain} of the blender $K$. 

\begin{figure}[h]
    \centering
       \includegraphics[width=7cm]{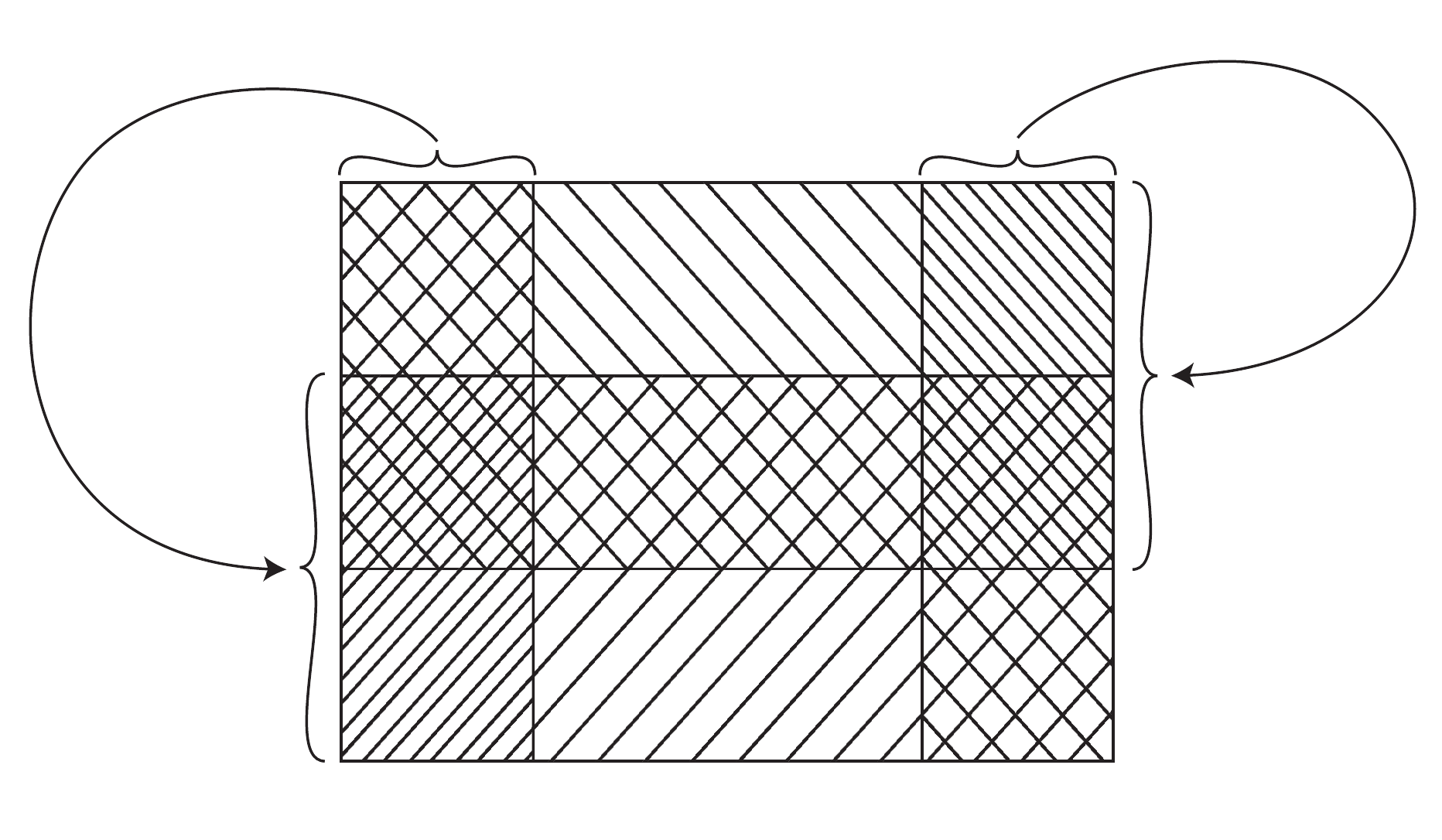}
    \caption{A blender of a surface.}
\end{figure}

Blender were discovered in \cite{BD96}, and then used in \cite{BD99,DNP} to produce a locally generic set of diffeomorphisms displaying infinitely many sinks. 
In \cite{BE15}, the notion of blender has been adapted to local surface diffeomorphisms to produce a locally generic set of surface local diffeomorphism displaying infinitely many sinks following a similar argument to \cite{DNP}.

\paragraph{Area contracting saddle point} A surface local-diffeomorphism has an area contracting fixed point $P$ if the product of the stable and the unstable eigenvalues of $P$ has a modulus less than $1$.

\paragraph{Projectively hyperbolic source}  
A fixed point $S$ of a surface diffeomorphism $f$ is a \emph{projectively hyperbolic source} if $D_Sf$ has two  eigenvalues 
$\sigma_{uu}, \sigma_{u}$ with different moduli $1< |\sigma_{u}|<|\sigma_{uu}|$. 
The eigenspace associated to $\sigma_u$ is called the \emph{weak unstable direction}, whereas the eigenspace $E^{uu}(S)$ associated to $\sigma_{uu}$ is called the \emph{strong unstable direction}. 

A \emph{basin} of the source $S$ is an open neighborhood  $B$  of $S$ on which an inverse branch $g$ of $f$ is well defined and whose points $z\in B$ satisfy $g^n(z)\to S$.  Then $E^{uu}(S)$ extends continuously to a line field on $B$, denoted also by $E^{uu}(S)$ and so that for every $z\in B$, 
\[ \lim_{n\to +\infty}\frac1n \log (\|D_zg^n|E^{uu}(S)\|\cdot  |\sigma^{uu}|^n) \to 0\; .\]
The line field  is uniquely defined once $g$ is fixed, and there is a unique inverse branch of $g\colon B\to B$ which fixes $S$. Hence  $E^{uu}(S)$ is uniquely defined once $S$ and $B$ are fixed.  

If $S$ is a source of period $p$, then the above definitions and notations are canonically generalized by considering $f^p$ instead of $f$.  

Moreover it is well known that the line field $E^{uu}(S)$ is the tangent space of a unique $C^0$-foliation $\mathcal F^{uu}(S)$ on $B$, whose leaves are as regular as the dynamics \cite{Yoccozintro}. Note that $\mathcal F^{uu}(S)$ is uniquely defined once $S$ and $B$ are fixed.

A $C^1$-embedded curve $\Gamma$  in $B$ has a \emph{robust tangency} with  $\mathcal F^{uu}$ if any $C^1$-perturbation of $\Gamma$ has a tangency with one leaf of $\mathcal F^{uu}$. 
\paragraph{Hyperbolic sets for  families of dynamics}
Let us fix $k\ge 0$, $1\le d\le r\le \infty$, $X\in \{A,PS\}$ and a family of local diffeomorphisms $(f_a)_a\in C^{d,r}_X(\I^k,M,M)$, with $\I=[-1,1]$. 

It is well known that if $f_0$ has a hyperbolic fixed point $P_0$, then it persists for every $a$ small as a hyperbolic fixed point $P_a$, and the map $a\mapsto P_a$ is of class $C^d$.

 More generally, if  $K$ is a hyperbolic set for $f_0$, it persists for every $a$ small, but if the map $f_0|K$ is not bijective, we need to consider the space of preorbits $\arr K=\{(z_i)_{i\le 0}: f_0(z_i)=z_{i+1} \forall i<0\}$ of $K$. Let $\arr f_0$ be the shift map on $\arr K$. 
 \begin{thm}[Prop 1.6 \cite{BE15, BE152}]
For every $a$ in a neighborhood $V$ of $0$, there exists a map $h_a \in C^0(\arr K; M)$ so that:
\begin{itemize}
\item $h_0$ is the zero-coordinate projection $ (z_i)_i\mapsto z_0$.
\item $f_a \circ h_a= h_a\circ \arr {f_0}$ for every $a\in V$. 
\item For every $\arr  z\in \arr K$, the map $a\in V\mapsto h_a(\arr z)$ is of class $C^d$.
\end{itemize}
\end{thm}

The point $h_a(\arr z)$ is called the \emph{hyperbolic continuation} of $\arr z$ for $f_a$. We denote $z_a\in M$ the zero-coordinate of $h_a(\arr z)$. 
 The family of sets $(K_a)_a$, with $K_a:= \{z_a : \arr z \in \arr K\}$, is called the hyperbolic continuation of $K$.

The local stable and unstable manifolds $W^s_{loc} (z; f_a) $ and $W^u_{loc} (\arr z; f_a) $  are canonically chosen so that they depend continuously on $a$, $z$ and $\arr z$. 
 They are called the \emph{hyperbolic continuations} of $W^s_{loc} (z; f_0) $ and $W^u_{loc} (\arr z; f_0) $ for $f_a$. 
Let us recall:
\begin{prop}[Prop 1.6 \cite{BE15,BE152}]\label{Wupara}
For every $z\in  K$, the family $(W^s_{loc} ( z; f_a))_{a\in V}$ is of class $C^{d,r}_A$.
For every $\arr z\in \arr K$, the family $(W^u_{loc} (\arr z; f_a))_{a\in V}$ is of class $C^{d,r}_A$.
Both vary continuously with $z\in K$ and $\arr z\in \arr K$.
 \end{prop}

In order to study the bifurcation between stable, unstable and strong unstable manifolds for a family $(f_a)_a$ of dynamics $f_a$, it is natural to study the action of $(f_a)_a$ on \emph{$C^d$-jets}. 
\label{notationJdM}
Given a $C^d$-family of points $(z_a)_{a\in \R^k }$, its \emph{$C^d$-jet at $a_0\in \R^k$} is  $J^d_{a_0} (z_a)_a= \sum_{j=0}^d \frac{\partial^j_a z_a}{j!}  a^{\otimes j}$.  Let $J^d_{a_0}(\R^k,M)$ be the space of $C^d$-jets of $C^d$-families of points in $M$ at $a=a_0$. 

We notice that any $C^{d,r}_A$-family $(f_a)_a$ of $C^r$-maps $f_a$ of $M$ acts canonically on $J^d_{a_0} M$ as the map:
\[J^d_{a_0} (f_a)_a\colon  J^d_{a_0}(z_a)_a\in J^d_{a_0}(\R^k,M)\mapsto 
J^d_{a_0} (f_a(z_a))_a\in J^d_{a_0}(\R^k, M)\]
  
\begin{rema}Suppose that $M$ is a surface. If $f_{a_0}$ has a hyperbolic fixed point $P$ with stable and unstable eigenvalues $\lambda_s,\lambda_u$ then 
$J^d_{a_0} (P_a)_a$ is the unique hyperbolic fixed point of $J^d_{a_0} (f_a)_a$. Moreover  the stable and unstable directions of $D_{J^d_{a_0} (P_a)_a} J^d_{a_0} (f_a)_a$ have the same dimension. The restriction of $D_{J^d_{a_0} (P_a)_a} J^d_{a_0} (f_a)_a$ to each of these spaces is the composition of $\lambda_s id$ (resp. $\lambda_u id$) with a nilpotent map. We observe that $W^s(J^d_{a_0}(P_a)_a)$ consists of $C^d$-jets of families $(Q_a)_a$ so that $Q_a$ is in $W^s(P_a; f_a)$ for every $a$.  
\end{rema}

More generaly, given a hyperbolic set $K$ for $f_{a_0}$, 
the set $J^d_{a_0} (K_a)_a:= \{J^d_{a_0}(h_a(\arr z))_a: \arr z\in \arr K\}$ is a hyperbolic compact set for  $(J^d_{a_0}(f_a))_a$. 
   
The first example of parablender was given in \cite{BE15}; in \cite{BCP16} a new example of parablender was given and therein the definition of parablender was formulated as:
\begin{defi}[$C^d$-Parablender]
A family $(K_a)_a$ of blenders for $(f_a)_a$  is a \emph{$C^d$-parablender} at $a={a_0}$ if the following condition is satisfied.
There exists a non-empty open set  $O$ of $C^d(\R^k, M)$ so that for every $(f'_a)_a$ $C^{d,d}_A$-close to  $(f_a)_a$,  for every $\gamma\in O$, there exist $\arr z\in \arr K$ and a $C^d$-family  $(Q_a)_a$ of points in a continuous family $(W^u_{loc}(\arr z; f'_a))_a$ of local unstable manifolds  satisfying:
\[d(\gamma(a), Q_a)= o(\|a-a_0\|^d)\; .\]

The open set $O$ is called a \emph{covered domain} for the $C^d$-parablender $(K_a)_a$. 
\end{defi}
\begin{rema}
We notice that if $J^d(K_a)_a$ is a blender for $J^d_{a_0}(f_a)_a$ then   $(K_a)_a$ is a $C^d$-parablender at $a_0$ for $(f_a)_a$. We do not know if it is a necessary condition. 
\end{rema}
\begin{exam}[$C^d$-Parablender]\label{expparablender} We propose here a small variation of Example 2.2 \cite{BE15}. Let $\Delta:=  \{-1,1\}^E$ with $E:= \{i=(i_1,\dots, i_k)\in \{0,\dots, d\}^k: i_1+\cdots +i_k\le d\}$.  Each  $\delta \in \Delta$ is seen as a function $\delta\colon i\in E\mapsto \delta(i)\in \{-1,1\}$.

Consider ${Card\, \Delta}$ disjoint segments $D:= \sqcup _{
\delta \in \Delta} I_\delta $ of $[-1/2,1/2]\setminus \{0\}$.

Let $Q\colon  \sqcup _{\delta\in \Delta} I_\delta \to [1,1]$ 
be a locally affine, orientation preserving map which sends each $I_\delta$ onto $[-1,1]$. Let $(\mathring f_a)_a$ be the $k$-parameter family defined by:
\[\mathring f_a(x,y) \colon(x,y)\in D\times [-3,3]\longmapsto \begin{array}{cc}
(Q(x), \frac 23 y  + \sum_{i\in E } \delta(i)\cdot a_{1}^{i_1}\cdots a_{k}^{i_k}) & \text{if } x\in I_\delta \; .
\end{array} \]
In the above definition we use the convention $0^0=1$. Then $(f_a)_a$ is a $C^\infty$-family of $C^\infty$-local diffeomorphisms. 

We notice that the maximal invariant set of $\mathring f_0$ is a blender $K$.

Let us define the following subset of $J^d_{0} \R^2$, with $a^i= \prod_{j=1}^k a_j^{i_j}$ for every $a\in \R^k$ and $j\in E$: 
\[\hat O:= \{\sum_{i\in E} (x_i,y_i) \cdot a^i
: |x_i|< 1,|y_i|< 2\}\quad,\quad 
\hat O':= \{\sum_{i\in E} (x_i,y_i) \cdot a^i
 : |x_i|\le 1/2,|y_i|\le 1\}\]
\[ \text{and}\quad \hat O_\delta:= \{\sum_{i\in E} (x_i,y_i) \cdot a^i
: |x_i|< 1, 0\le \delta(i)\cdot y_i< 2\}\, .\]
We observe that $\hat O=\cup_{\delta\in \Delta} \hat O_\delta=\hat O \Supset \hat O'$. Also for every $\delta \in \Delta$, 
Let $g_{a\; \delta}$ be the inverse of  $\mathring f_a| I_\delta\times [-3,3]$. Both are product dynamics of intervals. 

Let us show that  $J^d_0 (g_{a\; \delta})_a$  sends $\hat O_\delta$ into $\hat O'$.
The map $J^d_0 (\mathring f_a)_a$ is the composition of a hyperbolic map with a translation of by the $C^d$-Jet $\sum_E \delta(i) a_1^{i_1}\cdots a_k^{i_k}$. Thus $J^d_0 (g_{a\; \delta})_a$ is a composition of:
\begin{itemize}
\item a translation which sends $\hat O_\delta$ to
\[\{\sum_{i\in E} (x_i,y_i) \cdot a^i
: |x_i|< 2, -1\le \delta(i)\cdot y_i< 1\}\, .\]
\item a hyperbolic transformation which sends the latter into  $\hat O'$.
\end{itemize}


For every $\u \delta=(\delta_{i})_{i\le -1}\in \Delta^{\Z^-}$, for every $(f'_a)_a$ $C^{d,d}_A$-close to $(\mathring f_a)_a$ we define $W^u_{loc}(\u \delta; f'_a):= \cap_{n\ge 1} f_a'^n(I_\delta \times [-3,3])$. We notice  that $(W^u_{loc}(\u \delta; f'_a))_{\arr \delta \in \Delta^{\Z^-}}$ is a continuous family of unstable manifolds of the hyperbolic continuation $K'_a$ of $K$. 

Let $O$ be the set of families $\gamma$ of points so that $J^d_0 \gamma\in \hat O$. 
We notice that  there exists $\alpha>0$, so that given $(f_a)_a$ $C^{d,d}_A$-close to $(\mathring f_a)_a$, for all  $\|a_0\|\le \sqrt[k]{ 2}\alpha$  and $\gamma\in O$, the following property holds true:
There exists a sequence of preimages $(\gamma^{i})_{i\le -1}$  of $\gamma$ and symbols $\u \delta =(\delta(i))_{i\le -1}$ so that $\gamma^0=\gamma$, $J^d_{a_0} \gamma^{i+1}$ is in $\hat O_{\delta^i}$, $\gamma^i(a)= (f'_a| I_{\delta(i)} \times [-3,3])^{-1}(\gamma^{i+1}(a))$,  and $J^d_{a_0} \gamma^i$ is in $\hat O'$ for every $i\le -1$.
By proceeding like in Theorem B \cite{BCP16}, this implies the existence of a $C^d$-family $(Q_a)_a$ of points in $(W^u_{loc}(\u \delta; f'_a))_a$ so that $J^d_{a_0} (Q_a)_a = J^d_{a_0} (\gamma(a))_a$.

%

As the family $(\mathring f_{a+a_0})_a$ is close to $(\mathring f_{a})_a$ for every $a_0$ small, there exists $\alpha>0$ small such that $(K_a)_a$ is a $C^d$-parablender at every $\|a_0\|\le \sqrt[k]{ 2}\alpha$.

 Consequently the family $(f_a)_a:=(\mathring f_{\alpha \cdot a})$ displays the $C^d$-parablender $(K_a)_a=(\mathring K_{\alpha \cdot a})_a$ at every $a_0\in \I^k$, with covered domain containing every constant family of points in $[-1/2,1/2]\times [-1,1]$.

\end{exam}

\subsection{Statement of the main Theorem}
\label{Statement of the main theorem}
  
 Let $\mathcal U$ be the open set of $C^1$-local surface diffeomorphisms of a surface $M$ which have a blender $K$, an area contracting saddle fixed point $P$ and a projectively hyperbolic source $S$ with a strong unstable foliation $(\mathcal F^{uu},B)$ so that :
\begin{enumerate}[$(i)$]
\item[$(H_0)$] $S$ is in a domain 
robustly covered by a continuous family $(W^u_{loc}(\arr z; f))_{\arr z\in \arr K}$ of local unstable manifolds of the blender $K$.
\item[$(H_1)$] $K$ is included in $B$  and the  stable direction of $K$ is not tangent to $\mathcal F^{uu}$. 
\item[$(H_2)$] A segment of  $W^s(P; f)$  has a robust tangency with $\mathcal F^{uu} $ and $W^u(P;f)$ has a transverse intersection with $W^s(K;f)$.
\end{enumerate}
\begin{figure}[h]
    \centering
        \includegraphics[width=10cm]{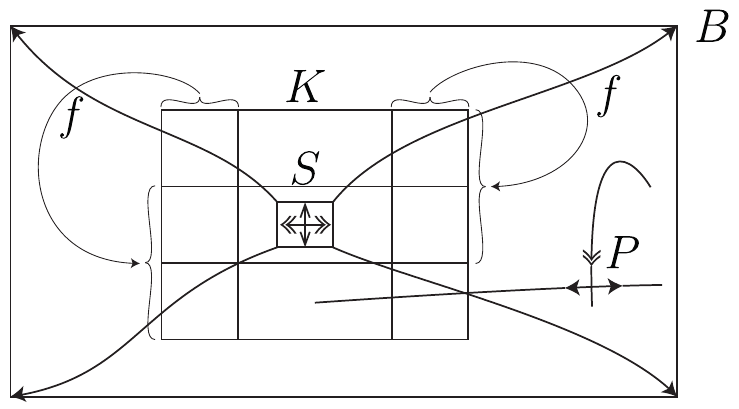}
\end{figure}
It will be clear after reading the sketch of proof of the main theorem that a $C^r$-generic diffeomorphism in $\mathcal U$ displays infinitely many sinks for every $\infty\ge r\ge 1$. See also \cite{New80,Asa08} for a parameter free argument.  The parametric version of this result  involves more hypotheses. 

Let  $(f_a)_a$ be a $C^{d,d}$-family of maps $f_a$ in $\mathcal U$ 
and denote by $K_a$, $P_a$ and $S_a$ the hyperbolic continuations of respectively  the blender, the saddle point and the source with which $f_a$ satisfies $(H_0-H_1-H_2)$ for every $a\in \I^k=[-1,1]^k$. Assume
that $(K_a)_a$ is a $C^d$-parablender at every $a_0\in \mathbb I^k$, and its covered domain  contains $J^d_{a_0} (S_a)_a$. More specifically there exists a continuous family of unstable manifolds $((W^u_{loc}(\arr  k; f_a))_a)_{\arr k\in \arr K}$ so that for every  $C^{d,d}$-perturbation $(f'_a)_a$ of $(f_a)_a$, with $(S'_a)_a$ the hyperbolic continuation of  $(S_a)_a$ it holds:
   \begin{itemize}
\item[$(H_3)$] For every $a_0\in \mathbb I^k$, there exists a $C^d$-family $(Q_a)_a$ of points $Q_a$ in $W^u_{loc} (\arr k; f_a')$ so that 
 $d(Q_a, S'_a)= o(\| a-a_0\|^d)$. In particular  $W^u_{loc}(\arr z;f'_{a_0})$  contains $S'_{a_0}$ .
 \item[$(H_4)$] $W^u_{loc}(\arr z ;f'_{a_0})$ is not tangent to the weak unstable direction of $S'_{a_0}$.
\end{itemize}  
%
%
%
%
Note that $(H_3)$ \& $(H_4)$ imply $(H_0)$. 

Let $\mathcal U^{d,d}$ be the open set of $C^{d,d}_A$-families of maps satisfying $(H_0-H_1-\cdots-H_4)$. 
We observe that for every $X\in \{A,PS\}$, and $1\le d\le r\le \infty$, the set $ {\mathcal U}_{X}^{d,r}:= {\mathcal U}_A^{d,d}\cap C^{d,r}_X(\I^k,M,M)$ is open for the  $ C^{d,r}_X$-topology. 
We recall that $\I^k:=[-1,1]^k$.

 \begin{theo}[Main theorem]\label{main}
For every $1\le k<\infty$, any topology $X\in \{A,PS\}$,  $1\le d\le r\le \infty$ with $d< \infty$, there exists a $C^{d,r}_X$-Baire generic set $\mathcal R$ in $ {\mathcal U}_{X}^{d,r}$ so that for every $(f_a)_a\in \mathcal R$ and every $a\in \I^k$, the map $f_a$ displays infinitely many sinks.
\end{theo}
   
   \begin{exam}\label{examfonda}
Here is a variation of \textsection 4 of \cite{BE15} in which  we add a source. 

 Let  $(f_a)_{a\in \I^k} \colon D\times [-3,3]\to [-1,1]\times  [-4,4]$ be the map of example \ref{expparablender} exhibiting a $C^d$-parablender $(K_a)_a$ at every $a\in \I^k$ and the constant family $(0)_{a\in \I ^k}$ in its covered domain.
  We recall that $D$ is formed by ${Card\, \Delta}$-intervals of $(-1,1)\setminus \{0\}$.   
Let $I_S, I_P,I_{P'}\subset (-1,1) \setminus D$ be disjoint segments, so that $I_S$ is centered at $0$. We extend $Q$ to 
$D\sqcup I_S\sqcup  I_P\sqcup  I_{P'}$ so that $Q$ remains locally affine and orientation preserving, and sends as well  $I_S$, $  I_P$, $I_{P'}$ onto $[1,1]$. 
Let $x_S\in I_S$ and $x_P\in I_P$ be  fixed points of $Q$. Let $x_{P'}$ be the preimage by $Q|I_{P'}$ of $x_P$.  Let:
\[f_a(x,y) \colon(x,y)\in (D\sqcup I_S\sqcup  I_P\sqcup  I_{P'})
\times [-3,3]\mapsto \left\{\begin{array}{cc}
f_a(x,y) & \text{if } x\in D\\
(Q(x), \frac{y}{\sqrt{|I_S|}}) & \text{if } x\in I_S, \\
(Q(x),  |I_P|^2 y) & \text{if } x\in I_P, \\
(y -(x-x_{P'})^2+x_P,x_{P'}-x) & \text{if } x\in I_{P'}, \\
\end{array}\right. \]
With $\hat \R$ the one point compactification of $\R$, since $Q$ is orientation preserving, it is  easy to extend $f$ to a local diffeomomorphism of the torus $\hat \R^2$ of degree $Card\, \Delta+3$.

We notice that $S=(0,0)$ is a projectively hyperbolic source $S$ with vertical weak unstable direction, hence transverse to the local unstable manifold of $(K_a)_a$ (which are of the form $([-1,1]\times \{y_a\})_a$).

 Note also that the hyperbolic continuation of $S$ is the constant family $(0)_a$ and so belongs to the covered domain of the  $C^d$-parablender $(K_a)_a$ at every $a_0\in \I^k$.

Also $P=(x_P,0)$ is an area contracting saddle fixed point, with vertical local stable manifold. The preimage of this local stable manifold in $I_{P'}\times [-1,1]$ is the graph of the function $x\mapsto (x-x_{P'})^2$. It has a robust tangency with $\mathcal F^{uu}(S)$ whose leaves are all horizontal. 

Consequently  $(f_a)_{a\in \I^k}$ belongs to $\mathcal U^{d,d}_A\cap C^{\infty,\infty}(\I^k, M,M)$.

Hence by Theorem \ref{main}, for every $\infty\ge r\ge d$ and $X\in \{A,PS\}$, a $C^{d,r}_X$-Baire generic perturbation of $(f_a)_a$ displays infinitely many sinks at every parameter $a\in \I^k$. 
\end{exam}

A corollary of the above example and of the proof of the Main theorem is:
\begin{coro}\label{theo2}
For every compact manifold of dimension $\ge 3$, for all $\infty > r\ge d\ge 1$, $\infty >k\ge 0$, and $X\in \{PS,A\}$, there exists an open set  $\hat U$ in $C^{d,r}_X(\I^k, M,M)$ of families $(\hat f_a)_a$ of \emph{diffeomormphisms} $\hat f_a\in Diff^r(M)$   and a Baire residual set $\mathcal R$ in  $\hat U$ so that for every $(f_a)_a\in \mathcal R$, for every $a\in \I_k $, the map $f_a$ displays infinitely many sinks.
\end{coro}
The proof will be done in section \ref{diffcase}. We will add 
an extension to the argument of the corresponding corollary of \cite{BE15}, in order to cover the new classes of regularity.

\section{Sketch of proof}
Let $k\ge 1$, $r\ge d\ge 1$ with $d<\infty$   and $X\in \{A,PS\}$. 
To avoid technical difficulties, we will work only with $C^\infty$-families in ${\mathcal U}^{d,r}_X$, which are families in   ${\mathcal U}^\infty $, where:
\[ {\mathcal U}^\infty := {\mathcal U}^{d,d}_A\cap C^{\infty,\infty}(\I^k,M,M)= {\mathcal U}^{d,r}_X\cap C^{\infty,\infty}(\I^k,M,M),\quad \I^k=[-1,1]^k\]

 We observe that ${\mathcal U}^\infty$ is a dense set in ${\mathcal U}^{d,r}_X$ for the  $C^{d,r}_X$-topology. It is also dense in the following space:
\[ {\mathcal U}^{d,\infty} := {\mathcal U}^{d,d}_A\cap C^{d,\infty}(\I^k,M,M)= {\mathcal U}^{d,r}_X\cap C^{d,\infty}(\I^k,M,M)\] 
 
 The following lemma enables us to work only with the spaces 
${\mathcal U}^{d,\infty}$ and ${\mathcal U}^{\infty}$.
 
 \begin{lemm}\label{lemmfonda}
The main Theorem holds true if for every $N>0$, there is a $C^{d,\infty}$-dense set in ${\mathcal U}^{d,\infty}$ of families $(f_a)_a\in \mathcal U^{\infty}$, so that for every $a\in \I^k$, the map $f_a$ displays  a sink of period at least $N$.
\end{lemm}
\begin{proof}
For every $N\ge 1$, the set $\mathcal V_{N}$ of families $(f_a)_a\in \mathcal U^{d,r}_X$ such that $f_a$ has a sink of period at least $N$ for every $a\in \I^k$ is open and dense in $\mathcal U^{d,r}_X$. Hence the following set is Baire residual in $ {\mathcal U}^{d,r}_X$:
$\mathcal R := \bigcap_{N\in \N} \mathcal V_{N}$.
We observe that for every $(f_a)_a\in \mathcal R$, for every $a\in \I^k$, the map $f_a$ has a sink of arbitrarily large period. Hence $f_a$ has infinitely many sinks.
\end{proof}
\medskip

We recall that by $(H_3)$, for any $(f_a)_a\in {\mathcal U}^\infty $ and any $a_0\in \I^k $ the source $(S_a)_a$ has its $C^d$-jets at $a=a_0$ in the covered domain of the $C^d$-parablender $(K_a)_a$. This means that there exists $\arr Q\in \arr K$ and a $C^\infty$-family $(Q'_a)_a$ of points  in the local unstable manifold $(W^u_{loc}(\arr Q;f_a))_a$ so that $d(Q'_a,S_a)= o(\|a-a_0\|^d)$. Hence a $C^{d,\infty}$-perturbation of $(f_a)_a$ puts $S_a$ at $Q'_a$  for all $a$ in a compact neighborhood $V$ of $a_0$. 
We observe that the maximal $\eta$-ball centered at $a_0$ and contained in $V$ is bounded from below only by the size of the allowed perturbation. We are going to keep $V$ intact in all the next steps.

By the second part of $(H_2)$, the unstable manifold of $P_a$ has a transverse intersection with a stable manifold of $K_a$. Hence by the following parametrized inclination lemma, $(W^u(P_a; f_a))_a$ accumulates on  $(W^u_{loc}(\arr Q; f_a))_a$.

\begin{lemm}[Parametrized inclination Lemma 1.7 \cite{BE15} ]\label{inclilemma}
Let $(f_a)_a$ be a smooth family of  local diffeomorphisms leaving invariant a hyperbolic compact set $(K_a)_a$ with unstable direction of dimension $d_u$. 
Let $(C_a)_a$ be a smooth family of submanifolds of dimension $d_u$ and intersecting transversally a local stable manifold of $K_a$. 

Then, for any $\arr Q\in \arr K$, for any local unstable manifold $(W^u_{loc}(\arr Q; f_a))_a$, there exists $C'_a\subset C_a$ so that the family  $(f_a^n(C'_a))_a$ is $C^\infty$-close to $(W^u_{loc}(\arr Q; f_a))_a$ when $n$ is large.
\end{lemm}

%
Hence by the second part of $(H_2)$ and the parametrized inclination lemma, we can assume that after an arbitrarily small $C^{d, \infty}$-perturbation, the family $(f_a)_a$ is of class  $C^\infty$ and there exists a segment a segment $\Gamma_a^u$  of $W^u(P_a; f_a)$ which contains $S_a$ for every $a\in V$.

By the first part of $(H_2)$, we can perturb $(f_a)_a$ so that a segment of $W^s(P_a;f_{a_0})$ has a quadratic tangency with a leaf of $\mathcal F^{uu}(S_{a_0})$ for every $a\in U$, where $U$ is equal to $V$ without a (possibly empty) $1$-codimensional submanifold. We recall that two curves of a surface have a \emph{quadratic tangency} if they are tangent, and their curvatures are different at a tangency point. 

We are going to construct a perturbation of $(f_a)_a$ which displays a persistent  homoclinic tangency:

\begin{defi}[Persistent homoclinic tangency]
Let $(f_a)_{a\in \I^k}$ be a smooth family of surface local-diffeomorphisms and $(P_a)_{a\in \I^k}$ a saddle periodic point. 

The saddle point $P_a$ has a \emph{homoclinic tangency} if  $W^u(P_a; f_a)$ is   tangent to $W^s(P_a; f_a)$ at one point $H_a$.

The homoclinc tangency is \emph{persistent} for $a$ in an open subset $U\subset \I^k$, if there exist a smooth family $(\Gamma^u_a)_{a\in U}$ of embedded segments in $(W^u(P_a;f_a))_a$ and a smooth family of points $(H_a)_{a\in V}\in (\Gamma_a^u)_{a\in U}$ so that $W^s(P_a;f_a)$ is tangent to $\Gamma_a^u$ at $H_a$ for every $a\in U$. 

%

\end{defi}

For the sake of simplicity, let us assume that $V=U$ (the tangency of $W^u_{loc}(P_a; f_a)$ with $\mathcal F^{uu}$ is robustly quadratic) and that $V$ is compact. This extra hypothesis can be assumed when $d\ge 2$.

The following proposition implies that for every $a_0\in \I^k$, there exists a dense set in ${\mathcal U}^{d,\infty}$ of families $(f_a)_a\in {\mathcal U}^\infty$ so that $P_a$ has a persistent homoclinic tangency for $a$ in the compact neighborhood $V$ of $a_0$. 

  \begin{prop}\label{tangencycreation}
Let $V\subset \I^k$ be a compact subset. Let $(f_a)_{a\in \I^k} $ be a $C^\infty$-family of diffeomorphisms, which has a projectively hyperbolic source $(S_a)_a$. Let $(C_a)_a$ be a smooth family of  embedded curves $C_a$, so that for every $a\in V$, $C_a$  has a quadratic tangency with $\mathcal F^{uu}(S_a)$ at a point $H_a$ depending continuously on $a\in V$. 
Let $(W_a)_{a}$ be a smooth family of embedded curves so that  
 for every $a\in V$, $W_a$ contains $S_a$ in its interior and $T_{S_a} W_a$ is not equal to the weak unstable direction of $S_a$.  

Then there exists a smooth perturbation $(W'_a)_a$ of $(W_a)_a$ and $n\ge 0$ so that $f^{n}_a(W'_a)$ has a quadratic tangency with $C_a$ which persists for $a\in V$.  
\end{prop}
This proposition will be proved in section \ref{tangencycreationsec}.

Indeed by $(H_4)$ we can apply Proposition \ref{tangencycreation} with $W_a$ equal to the segment of the unstable manifold $\Gamma_a^u$ containing 
$S_a$ and $C_a$ the segment  of stable manifold of $P_a$ which is robustly tangent to $\mathcal F^{uu}(S_a)$, for every $a$ in the compact neighborhood $V$ of $a_0$. 

This implies that up to a $C^{d,\infty}$-perturbation, we can assume that the saddle point $P_a$ displays a robust homoclinic tangency for every $a\in V$. 

Then the following proposition implies that for every $N\ge 0$, up to a $C^{d,\infty}$-perturbation, the dynamics $f_a$ displays a sink of period at least $N$ for every $a$ in  $V$.
%
 \begin{prop}\label{sinkcreation}
Let $V\subset \I^k$ be a compact subset. Let $(f_a)_{a\in \I^k} $ be a $C^\infty$-family of local diffeomorphisms. We suppose that for every $a\in V$, the map $f_a$ has an area contracting saddle point $P_a$ which displays a  persistent homoclinic tangency at $H_a$ for  $a\in V$. Then for every $N\ge 1$ and $\eta>0$, there exists a smooth perturbation $(f'_a)_a$ such that for every $a\in V$ :
\begin{itemize}
\item  for every $z$ which is not in the $\eta$-ball $B(H_a,\eta)$, it holds $f_a(z)=f_a'(z)$,
\item the map $f_a$ has a sink of period at least $\ge N$.

\end{itemize}   
\end{prop} 
This proposition will be proved in section \ref{proofofsinkcreation}.

By Lemma \ref{lemmfonda}, to prove the main theorem, it suffices to construct for every $N\ge 0$,  a dense set  in ${\mathcal U}^{d,\infty}$ of smooth families which have a sink of period $\ge N$ for every $a\in \I^k$. This is stronger that what was sketched above: this was proved only for $a\in V$, where the size of $V$ was determined by the size of the perturbation allowed when  $(H_3)$ was used. 

To overcome this difficulty, our strategy is to replicate the source $S$ to others satisfying also $(H_2)$ and $(H_3)$. The replication process uses $(H_1)$ and $(H_4)$. It leads us to the following Proposition proved in  section \ref{proofpropfonda3}.

%
 
\begin{prop}\label{propfonda3}
There is a dense set in ${\mathcal U}^{d,\infty}$ formed by families  $(f_a)_a\in {\mathcal U}^\infty$ which satisfy the following property:

There exists a finite open covering $(U_i)_i$ of  $\I^k$ so that for every $i$ there exist a projectively hyperbolic  source $(S_{ia})_{a\in U_i}$ and a continuous of family of segments $(\Gamma^{u}_{ia})_a$ of $(W^u(P; f_a))_a$ such that:
\begin{enumerate}[$(i)$]
\item $S_{ia}$ is in $\Gamma_{ia}^{u}\subset W^u(P_a; f_a)$, for every $a\in U_i$.
\item $T_{S_{ia}} \Gamma_{ia}^{u}$ is not the weak unstable direction of $S_{ia}$,  for every $a\in U_i$.       
\item There exists a (smooth) family $(H_{ia})_{a\in U_i}$ of points  in $(W^s(P_a; f_a))_{a\in U_i}$ at which $W^s(P_a; f_a)$ has a quadratic tangency with the strong unstable foliation of $S_{ia}$,  for every $a\in U_i$.
\item For every $a\in U_i\cap U_j$ with $i\not= j$, the sets $(f^k_a(H_{ia}))_{k\ge 0}$ and   $(f^k_a(H_{ja}))_{k\ge 0}$ are disjoint and the orbits  $(f^k_a(S_{ia}))_{k\ge 0}$ and $(f^k_a(S_{ja}))_{k\ge 0}$ are disjoint. 
\end{enumerate}
\end{prop}

In section \ref{propfonda3main}, a development of the above sketched argument (where $(S_a)_a$ is replaced by $(S_{ia})_{a\in \mathcal U_i}$) will prove that Propositions \ref{tangencycreation}, \ref{sinkcreation} and \ref{propfonda3} imply Lemma \ref{lemmfonda}. We recall that  Lemma \ref{lemmfonda} implies the main theorem. 


\section{Tangency creation (Proof of Prop. \ref{tangencycreation})}
\label{tangencycreationsec}
In this section we consider a $C^\infty$-family $(f_a)_a$ of  diffeomorphisms of $\R^2$ which display a projectively hyperbolic source $S_a$ for every $a$ with strong unstable direction $E^{uu}_a$.

It is useful to regard the following smooth bundle automorphism over $f_a$:
\[T f_a\colon \R^2\times \mathbb P\R^1\to \R^2\times \mathbb P\R^1\] 
\[(z,L)\mapsto (f_a(z), D_zf_a(L))\]

We notice that the strong unstable direction $E^{uu}_a$ is a hyperbolic point for $T f_a$ with unstable direction $\R^2$ and stable direction the tangent space $\R$ of the second coordinate $\mathbb P\R^1$.

The following well known proposition is important: 
\begin{prop}\label{Fuusmooth}
The strong unstable  foliation $\mathcal F^{uu}(S_a)$ on the
 neighborhood of $S_a$  is of class $C^\infty$ and depends $C^\infty$ on $a\in \I^k$.
\end{prop}
\begin{proof}
Observe that $(S_a,E^{uu}(S_a))$ is a hyperbolic point of $Tf_a$ and that the tangent space of $\mathcal F^{uu}(S_a)$ is its local unstable manifold. Indeed, every vector $u$ tangent to $\mathcal F^{uu}(S_a)$ at a point $z$ displays a backward orbit by $Tf_a$ which converge to $(S_a,E^{uu}(S_a))$.

 As $T f_a$ is of class $C^\infty$, its local unstable manifold and so the foliation $\mathcal F^{uu}(S_a)$ are of class $C^\infty$. As the local unstable manifold  of a hyperbolic point of a smooth family of diffeomorphisms depends smoothly on the parameter by Prop. \ref{Wupara}, the foliation $\mathcal F^{uu}(S_a)$ depends smoothly on $a$.
%
%
%
\end{proof}

We are now ready to prove Proposition \ref{tangencycreation}.
We recall  that $(C_a)_a$ denotes a  $C^\infty$-family of  embedded curves $C_a$, and $V$ a compact set of $\I^k$ so that for every $a\in V$, the curve $C_a$  has a quadratic tangency with $\mathcal F^{uu}(S_a)$ at a point $H_a$ depending continuously on $a$. Also $(W_a)_{a}$   denotes a $C^\infty$-family of embedded curves $W_a$ which contains $S_a$  for every $a\in V$ and so that $T_{S_a} W_a$ is not equal to the weak unstable direction of $S_a$. 
We are going to show the existence of a $C^\infty$-perturbation of $(W'_a)_a$ of $(W_a)_a$ and $n\ge 0$ so that $f^{n}_a(W_a)$ displays  a quadratic tangency with $C_a$ for every $a\in V$.

%
%
\begin{figure}[h]
    \centering
        \includegraphics[width=10cm]{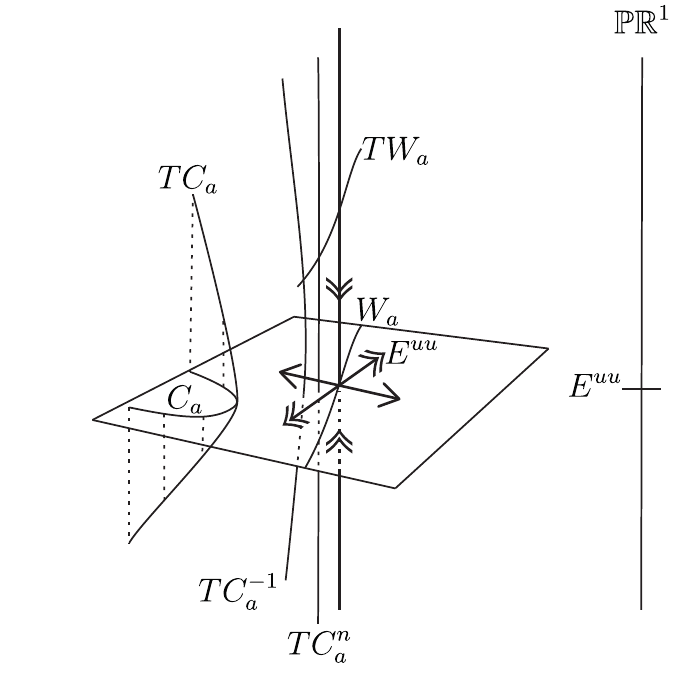}
\end{figure}

\begin{proof}[Proof of Proposition \ref{tangencycreation}]
We continue with the above notion by denoting $T f_a\colon \R^2\times \mathbb P\R^1\to \R^2\times \mathbb P\R^1$ the action of $(f_a)_a$ on the line bundle of $\R^2$. Let $E^{uu}_a\in \{S_a\}\times \mathbb P\R^1$ be the strong unstable direction of $S_a$.

We observe that the tangent bundle $TC_a$ of $C_a$ is an embedded curve $T C_a$ in $\R^2\times \mathbb P\R^1$. As $C_a$ is tangent at one point to $\mathcal F^{uu}(S_{a})$, the curve $T C_a$ intersects $W^u_{loc}(E^{uu}_a; T f_a)$. As the  tangency is quadratic, this intersection is transverse. 

Hence by the inclination lemma, the preimages  $(TC^n_a)_{n\le -1}$ of $T C_a$ by $T f_a$ accumulate on $W^s_{loc}(E^{uu}_a; T f_a)$, for the $C^\infty$-topology. 
By the parametric inclination lemma  \ref{inclilemma}, the preimages  $((T C^n_a)_a)_{n\le -1}$ of $T C_a$ by $(T f_a)_a$ accumulate on $(W^s_{loc}(E^{uu}_a; T f_a))_a$, for the $C^\infty$-topology. 
 
Remark that $W^s(E^{uu}_a,Tf_a)$ is $\{S_a\}\times \mathbb P\R^1$ without the weak unstable direction of $S_a$.

Note also that the tangent bundle $TW_a$  of $W_a$ is an embedded curve of $\R^2\times \mathbb P\R^1$. The curve $W_a$ contains $S_a$ and is not tangent to the weak unstable direction of $S_a$, hence $TW_a$ intersects $W^s_{loc}(E^{uu}_a;Tf_a)$ at $T_{S_a}W_a=:(S_a,v_a)$. We observe that $a\mapsto v_a\in \mathbb P\R^1$ is of class $C^\infty$. 

As $ (T C^n_a)_a $ contains a segment close to $W^s_{loc}(E^{uu}_a;Tf_a)$, there exists a point $(x_a,y_a)\in\R^2$ so that $(x^n_a,y^n_a,v_a)_a$ is in $ (T C^n_a)_a $ and is $C^\infty$-close to   $(S_a,v_a)_a$.  

Consequently, there exists $(z^n_a)_a$ $C^\infty$-small so that 
$T C^n_a$ intersects $TW_a+(z^n_a,0)= T(W_a+z^n_a)$. 
Let  $\rho$ be a compactly supported bump function equal to $1$ on $V$. We notice that the proposition is proved with the perturbation $W_a'= W_a+\rho(a)\cdot z^n_a$ which is small for $n\le 0$ large. 
\end{proof}
\begin{rema}\label{tangencycreation2} In view of the above proof, the  source $(S_a)_a$ of Proposition \ref{tangencycreation} does not need to be fixed: it can be a periodic source. Then the same conclusion holds true. 
\end{rema}

  \begin{rema}\label{rematechnique} Moreover, the tangency point $\tilde H^{n}_a$ of $C_a^{n}$ with $W'_a$ satisfies that $(f^k_a ( \tilde H^{n}_a))_{k=n}^0$ is close to $(f_a^{k}(H_a))_{k=n}^0$, where $n\le 0$ is defined in the proof. 
  
Indeed the curve $TC_a^{n}$ being close to be vertical, the point $\tilde H^{n}_a$ is close to $f_a^{n}(H_a)$. Also the curve 
  $TC_a^{n}$ is contracted by $Tf^k _a$ for  every $a$ and for every $n\le k\le 0$, since $TC_a^{n}$ is close to the local stable manifold $W^s_{loc}(E^{uu}_a;Tf_a)$.
  \end{rema}
  
  \section{Sinks Creation (proof of Prop. \ref{sinkcreation})} 
\label{proofofsinkcreation}
The following section is devoted to the proof of Proposition \ref{sinkcreation2} below which implies Proposition \ref{sinkcreation}

Let $V\subset \R^k$ be a compact subset. For every $\eta>0$ small, let $V_\eta$ be the $\eta$-neighborhood of $V$. 
 Let $(f_a)_{a\in \R^k} $ be a $C^\infty$-family of local diffeomorphisms. We suppose that for every $a\in V$, the map $f_a$ has an area contracting saddle point $P_a$ which displays a  persistent homoclinic tangency at $H_a$ for  $a\in V$.

In other words, the $C^\infty$-family  $(H_a)_a$ is formed by the tangency points of the local stable manifolds $(W^s_{loc}(P_a; f_a))_a$ with a smooth family $(\Gamma_a^u)_a$ of embedded segments in $(W^u(P_a;f_a))_a$.
Let $(H^{-j}_a)_{j\le 0}$  be the $j^{th}$-preimage of $H_a$ defined using the inverse branches defining $\Gamma_a^u$. We observe that $H^0_a=H_a$ and that this presequence converges to $P_a$. We define:

\[\mathcal O(H_a)= \{H^{-j}_a : {j\le 0}\}\cup\{f_a^j(H_a):j\ge 0\}.\]

 \begin{prop}\label{sinkcreation2}
 For every $M\ge 1$ and $\eta>0$, there exists a smooth perturbation $(f'_a)_a$ such that:
\begin{itemize}
\item  for every $a\notin V_\eta$ or $z\notin B(H_a,\eta)$, it holds $f_a(z)=f_a'(z)$,
\item for every the map $f_a$ has a sink  $A_a$  of period at least $\ge M$ and with orbit in the $\eta$-neighborhood of $\mathcal O(H_a)$.
\item the sinks $A_a$ depends smoothly on $a\in V$, and the family  $(A_a)_{a\in V}$ of points is $C^\infty$-close to $(H_a)_{a\in V}$.  
\end{itemize}   
\end{prop}
\begin{proof}[Proof of Proposition \ref{sinkcreation2}]
 The set $\mathcal O(H_a)$ is discrete and has a unique accumulation point at $P_a$.  Hence $\eta=2 \cdot  d(H_a, \mathcal O(H_a)\setminus \{H_a\})$ is positive. 

Let $\lambda_a$ and $\sigma_a$ be respectively the stable and unstable eigenvalues of $P_a$. As $P_a$ is area contracting it holds:
\[|\lambda_a\sigma_a|<1\; .\]

Via a smooth family of charts, for every $a\in V$, we identify a neighborhood $N_a$ of $H_a$ with  $[-1,1]^2$ so that $H_a$ is identified to $0$,  a segment of $\Gamma^u_a\subset W^u(P_a; f_a)$ to $[-1,1]\times \{0\}$ and a segment of $W^s(P_a; f_a)$ to the graph of $x\in [-1,1]\mapsto \rho_a(x)$ with $\rho_a(0)=0$ and $D_0\rho_a=0$ for every $a\in V$. 

 We notice that for every $a\in V$, $H_a^{-k}$ is $C^\infty$-close to $P_a$ for $k\ge 1$ large.

	\begin{figure}[h]
    \centering
        \includegraphics[width=10cm]{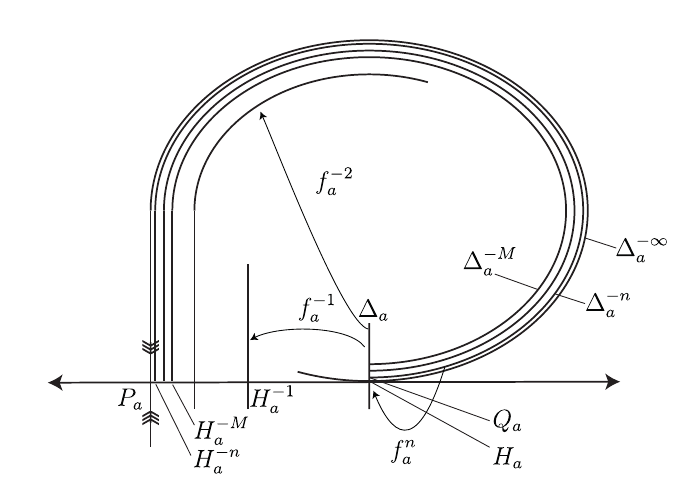}
    \caption{Preimages of $\Delta_a$.}
    \label{preimagedeltaa}
\end{figure}
As the map $f_a$ is in general not bijective, the stable manifold $W^s(P_a; f_a)$ is in general not connected.

Let us first assume that $H_a$ belongs to the component of $W^s(P_a; f_a)$ containing $P_a$.
Let $\Delta_a^{-\infty}$ be the segment of $W^s(P_a;f_a)$ with endpoints $P_a$ and $H_a$ (by assumption they belong to the same component of $W^s(P_a;f_a)$ ). 

Let $\Delta_a:= \{0\}\times [-\delta,\delta]$, with $\delta>0$ small.
By the inclination lemma, for $M$ large enough and $\delta>0$ small enough, for every $n\ge M$, the preimage of $\Delta_a$ by $f_a^n$  contains a segment $\Delta_a^{-n}$ bounded by $H_a^{-n}$ and $\Delta_a$, which is close to  $\Delta_a^{-\infty}$ (see figure \ref{preimagedeltaa}).  Actually, by the parametric inclinaison lemma,  the family of curves $(\Delta_a^{-n})_{a\in V}$ is $C^\infty$-close to $(\Delta_a^{-\infty})_{a\in V}$.

Let us fix $n\ge M$ large with respect to $M$. We observe that $\Delta^{-n}_a\cap [-1,1]^2$  is the graph of a function $x\in [0,1]\mapsto \hat \rho_a(x)$, with $(\hat \rho_a)_{a\in V}$ $C^\infty$-close to $(\rho_a)_{a\in V}$.

Let $Q_a$ be the unique endpoint of $\Delta_a^{-n}$ in $\Delta_a$.  We notice that  $(Q_a= (0,\hat \rho_a(0)))_{a\in V}$ is close to the constant family $(H_a=0= (0,\rho_a(0)))_{a\in V}$.

Let $Q'_a\in \Delta_a$ be the image of $Q_a$ by $f_a^n$. Note that $d(Q_a,Q'_a)<\delta<\!< \eta$. 
\begin{lemm}
For $n$ large, the family $(Q'_a)_{a\in V}$ is close to $ (H_a)_{a\in V}$.
\end{lemm}
\begin{proof}
We observe that $f_a|\Delta_a^{-\infty}$ is contracting with $P_a$ as unique fixed point. 

By identifying $\Delta^{-k}_a$ to $\Delta_a^{-\infty}$ for $k\ge M$ large (and so $H_a^{-k}$ to $P_a$),  the map  $f_a|\Delta_a^{-k}$ is contracting with $P_a$ as unique fixed point. 

As $(\Delta^{-k}_a)_{a\in V}$ is $C^\infty$-close to  $(\Delta^{-\infty}_a)_{a\in V}$ we have uniform bounds which enable us to prove that $(f^{n-M}_a(Q_a))_{a\in V}$ is $C^{\infty}$-close to $(H_a^{-M})_{a\in V}$ for $n$ large.  

By taking $n$ large compare to $M$, it comes that 
$(Q'_a)_{a\in V}= (f^{n}_a(Q_a))_{a\in V}$ is $C^{\infty}$-close to $(H_a)_{a\in V}$.\end{proof}

\begin{lemm}\label{LemmaQ}
The $n$-first iterates of $Q_a$ are in the $\eta$-neighborhood of $\mathcal O( H_a)$.   
\end{lemm}
\begin{proof}
By taking the notation of the previous proof, since $\delta>0$ is small with respect to  $\eta$, the $M$ first iterates of $Q_a$ are $\eta$-close to the $M$ first iterates of $H_a$. In the previous proof we saw that $\{f^k_a(Q_a): \; {n\ge k\ge M}\}$ is even closer to $\{ H^{-k}_a : \; 0\le k\le n-M\}$. 
\end{proof}

	\begin{figure}[h]
    \centering
        \includegraphics[width=6cm]{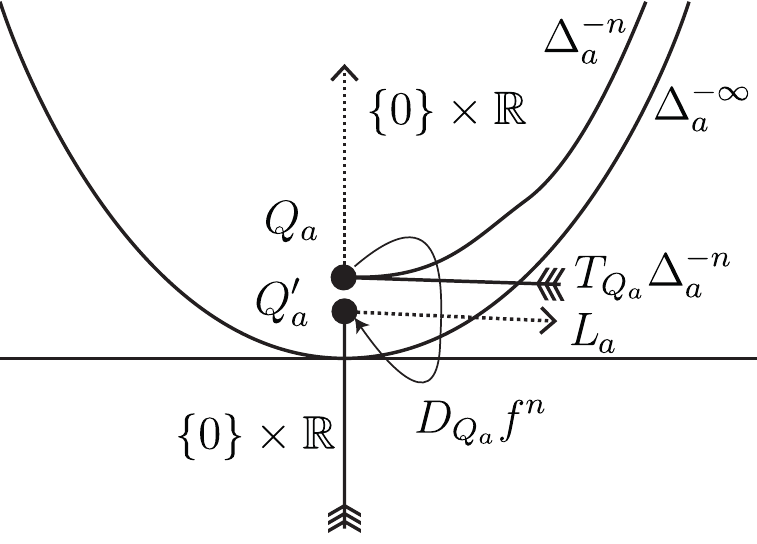}
    \caption{Notations involved.}
    \label{preuvemain}
\end{figure}
	
We recall that $f_a^n(Q_a)= Q'_a$ and $Df_a^n(T_{Q_a} \Delta^{-n}_a)= Q'_a+\{0\}\times \R$. Let $L_a \in \mathbb P \R^1$ be defined by $Q'_a+L_a:= Df_a^n(Q_a+\{0\}\times \R)$. By the inclination lemma, the lines family $(L_a)_{a\in V}$ is $C^\infty$- close to $(\R\times \{0\})_{a\in V}$. As $(Q_a)_{a\in V}$ is close to $(0)_{a\in V}$ and $(\Delta^{-n}_a)_{a\in V}$ is close to 
$(\Delta^{-\infty}_a)_{a\in V}$, the family $(T_{Q_a} \Delta^{-n}_a)_{a\in V}$ is close to $(T_{H_a} \Delta^{-\infty}_a)_a= (\R\times \{0\})_a$ and so to $(L_a)_{a\in V}$. 

Consequently, there exists a family $(f'_a)_a$ which is $C^\infty$-close to $(f_a)_a$ so that:
\begin{itemize}
\item $f'_a(Q_a')= f_a(Q_a)$ for every $a\in V$,
\item $Df'_a(Q_a'+L_a)= Df_a(T_{Q_a} \Delta^{-n}_a)$ for every $a\in V$,
\item $Df'_a(Q_a'+\{0\}\times \R)= Df_a(Q_a+\{0\}\times \R)$ for every $a\in V$.
\item $f_a'(z)=f_a(z)$ if $a\notin V_\eta$ or $z\notin B(H_a,\eta)$.
\end{itemize}

By Lemma \ref{LemmaQ}, the
 $n$-first $f_a$-iterates of $Q_a$ are included in the $\eta$ neighborhood of $\mathcal O(H_a)$.
Thus $\{f^k_a(Q_a): 1\le k\le n-1\}$  are $\eta$-distant to $H_a$. Thus it holds for every $a\in V$:
\[ D_{f'_a(Q'_a)} f'^{n-1}_a =   D_{f_a(Q_a)} f_a^{n-1}\]
Consequently the point  $Q'_a$ is $n$-periodic and $D_{Q'_a}f_a^n$ sends $L_a$ to $\{0\}\times \R$ with a contraction factor of the order of $\lambda_a^n$ and it sends $\{0\}\times \R$ to $L_a$ with an expansion factor of the order of $\sigma_a^n$. As $\lambda_a\sigma_a<1$, $D _{Q_a'}f_a^{2n}$ is a contracting homotety of factor of the order of $ (\lambda_a\sigma_a)^{2n}$. In particular $Q'_a$ is a sink of period $n\ge M$.      

If $H_a$ does not belong to the component of $W^s(P_a; f_a)$ containing $P_a$, then let $k\ge 0$ be minimal such that $f^k_a(H_a)$ belongs to the connected component of $W^s(P_a; f_a)$. By $f_a$-invariance of $W^u(P_a; f_a)$, this immersed submanifold is still tangent to $W^s(P_a; f_a)$ at $f^k_a(H_a)$. Hence by the above argument, for small perturbation of $W^u(P_a; f_a)$  around  $f^k_a(H_a)$ there is a persistent homoclinic tangency. The map $f_a^k$ being a diffeomorphism on a neighborhood of $H_a$, we can make this perturbation supported by a small neighborhood of $H_a$.

\end{proof}

\section{Proof that  Proposition \ref{propfonda3} implies  main Theorem \ref{main}}
\label{propfonda3main}

Let us assume that as in Proposition \ref{propfonda3}, there is a dense set in ${\mathcal U}^{d,\infty}$ formed by families  $(f_a)_a\in {\mathcal U}^\infty$ which satisfies the following property:

There exists a finite open covering $(U_i)_i$ of  $\I^k$ so that for every $i$ there exist a periodic, projectively hyperbolic  source $(S_{ia})_{a\in U_i}$ and a continuous family of embedded segments $(\Gamma^{u}_{ia})_a$ of $(W^u(P_a;f_a))_a$ such that:
\begin{enumerate}[$(i)$]
\item $S_{ia}$ is in $\Gamma_{ia}^{u}\subset W^u(P_a; f_a)$, for every $a\in U_i$.
\item $T_{S_{ia}} \Gamma_{ia}^{u}$ is not the weak unstable direction of $S_{ia}$,  for every $a\in U_i$.       
\item There exists a (smooth) family $(H_{ia})_{a\in U_i}$ of points  in $(W^s(P_a; f_a))_{a\in V_i}$ at which $W^s(P_a; f_a)$ has a quadratic tangency with the strong unstable foliation of $S_{ia}$,  for every $a\in U_i$.
\item For every $a\in U_i\cap U_j$ with $i\not= j$, the sets $(f^k_a(H_{ia}))_{k\ge 0}$ and   $(f^k_a(H_{ja}))_{k\ge 0}$ are disjoint and the finite orbits  $(f^k_a(S_{ia}))_{k\ge 0}$ and $(f^k_a(S_{ja}))_{k\ge 0}$ are disjoint. 
\end{enumerate}

We want to show that under these assumptions for every $M>0$, each of these families can be perturbed -- following the algorithm described in the sketch of proof --  to one which displays a sink of period at least $M$ for every $a\in \I^k$. 

In order to do so, we need to handle independently different perturbations. Hence we need to find some independent rooms in the phase space to do so. 

\medskip

For every $i$ and $a\in U_i$, we define the following sets of forward and backward orbits.

Put $\mathcal O^+(H_{ia}):= \{f^k_a(H_{ia}): k\ge 0\}$ and $\mathcal O^+(S_{ia}):= \{f^k_a(S_{ia}): k\ge 0\}$.  

Let $\mathcal O^{-}(H_{ia})$ be the preorbit of $H_{ia}$ defined thanks to the inverse branches defining the strong unstable foliation of $S_{ia}$.  We notice that $\mathcal O^{-}(H_{ia})$ accumulates on $\mathcal O^+(S_{ia})$. 
 We put: 
$$\mathcal O(H_{ia}):= \mathcal O^{+}(H_{ia})\cup \mathcal O^{-}(H_{ia}).$$

Also the segment $\Gamma_{ia}^{u}$ of $W^u(P_a; f_a)$  is defined by a sequence of inverse branches of $f_a$.  Let  $\mathcal O^{-}(S_{ia}):= (S^{(k)}_{ia})_{k\le -1}$ be the preobit of $S_{ia}$ associated to this sequence of inverse branches. We notice that $\mathcal O^{-}(S_{ia})$ is discrete with $P_a$ as unique accumulation point.  We put: 
$$\mathcal O(S_{ia}):= \mathcal O^{+}(S_{ia})\cup \mathcal O^{-}(S_{ia}).$$ 

\begin{fact}We can  assume that $\mathcal O^{+}(S_{ia})$ is disjoint from $\mathcal O^{-}(S_{ia})$.
\end{fact} 
\begin{proof}Indeed,  $\mathcal O^{+}(S_{ia})$ is finite since $S_{ia}$ is periodic thus  $\mathcal O^{-}(S_{ia})$ is not contained in $\mathcal O^{+}(S_{ia})$.  If the preimage $S_{ia}^{(-1)}$ of $S_{ia}$ is in $\mathcal O^{+}(S_{ia})$, then properties $(i-ii-iii-iv)$ are still satisfied by $S_{ia}^{(-1)}$ and the preimage $\Gamma_{ia}'^{u}$ of $\Gamma_{ia}^{u}$.
Consequently, 
by shifting the orbit, we can assume that $S_{ia}^{(-1)}$ is not in $\mathcal O^{+}(S_{ia})$, and so $\mathcal O^{+}(S_{ia})$ is disjoint from $\mathcal O^{-}(S_{ia})$.\end{proof}

By $(iv)$, for every $i\not= j$ such that $a\in U_i\cap U_j$,  the sets $\mathcal O^+(S_{ia})$ and $\mathcal O^+(S_{ja})$ are disjoint. This implies that the sets $\mathcal O(S_{ia})$ and $\mathcal O(S_{ja})$ are disjoint. 

We notice also that $\mathcal O(S_{ia})$ is disjoint from $\mathcal O(H_{ja})$, for every $i,j$ such that $a\in U_i\cap U_j$ (otherwise the source would converge to $P_a$). Consequently:
\begin{fact}
For every $a$ and all $i\not= j$ such that $a\in U_i\cap U_j$, 
the point $S^{(-1)}_{ia}$ is disjoint from the set $\mathcal O(S_{ja})\cup \mathcal O (H_{ja})\cup \mathcal O (H_{ia})\cup \{P_a\}$. 
\end{fact}

Note also that $\mathcal O(S_{ia})$ is discrete with a unique accumulation point at $P_{a}$ and  $\mathcal O(H_{ia})$ is discrete with  accumulation points $P_{a}$ and the finite set $\mathcal O^+(S_{ia})$. Hence the set $\mathcal O(S_{ja})\cup \mathcal O(S_{ia})\cup  \mathcal O (H_{ja})\cup \mathcal O (H_{ia})\cup \{P_a\}$ is compact and $S^{(-1)}_{ia}$ is isolated therein. 

 As these compact sets depend continuously on $a$, by shrinking slightly the covering $(U_i)_i$,  we obtain:
\begin{lemm}
There exists $\delta>0$ so that for every $i\not = j$ and $a\in U_i\cap U_j$, the point $S^{(-1)}_{ia}$ is at distance at least $3\delta>0 $ to $\mathcal O(S_{ja})\cup (\mathcal O(S_{ia})\setminus \{S^{(-1)}_{ia}\})\cup \mathcal O (H_{ja})\cup \mathcal O (H_{ia})$. 
\end{lemm} 

Let $(U'_i)_i$ be an open, relatively compact covering of $\I^k$ such that $cl(U'_i)\subset U_i,$ $\forall i$.  

By shrinking $\Gamma^{u}_{ia}$, we can assume its preimage $\Gamma^{u(-1)}_{ia}$ included in the $\delta/2$-neighborhood of $S^{(-1)}_{ia}$ for every $i$ and $a\in U_i$. By Proposition \ref{tangencycreation}, we can perturb $\Gamma^{u}_{ia}$ to a curve $\tilde \Gamma^{u}_{ia}$ which is tangent to $W^s(P_a; f_a)$ at a point $\tilde H_{ia}$.

By Remark \ref{rematechnique}, the forward orbit of $\mathcal O^+(\tilde H_{ia}):=\{f^k_a(\tilde H_{ia}): k\ge 0\}$ is included in the $\delta$-neighborhood of  $\mathcal O(H_{ia})$.

 As $S^{(-1)}_{ia}$ is $3\delta$-distant from $\mathcal O(H_{ia})\cup (\mathcal O(S_{ia})\setminus \{S^{(-1)}_{ia}\})$ we can handle so that:
 \begin{itemize}
 \item $f_{ia}(z) = f_a(z)$ if $a\notin U_i$ or $z\notin B(S^{(-1)}_{ia},\delta)$
 \item for every $a\in U'_i$, the curve
$\tilde \Gamma_{ia}^{u}$ is the hyperbolic continuation of  $\Gamma_{ia}^{u}$ for $f_{ia}$. 
\end{itemize}
This proves:

\begin{lemm}
For every $i$ there exists a smooth perturbation $(f_{ia})_{a\in U_i}$ of  $(f_a)_{a\in U_i}$ so that:
\begin{itemize}
\item  the hyperbolic continuation of $\tilde\Gamma_{ia}^{u}$ of $\Gamma_{ia}^{u}$ has a persistent homoclinic quadratic tangency with $W^s(P_a;f_{ia})$ at a point $\tilde H_{ia}\in \tilde \Gamma_{ia}^{u}\cap B(S_{ia}, \delta)$ and with forward orbit $\mathcal O^+(\tilde H_{ia})$ in the $\delta$-neighborhood of  $\mathcal O(H_{ia})$, for every $a\in U_i'$. 
\item $f_{ia}(z)=f_a(z)$ for every $z\notin B(S^{(-1)}_{ia},\delta)$ and $a\in U_i$. 
\end{itemize}
\end{lemm}

We define $\mathcal O^-(\tilde H_{ia})$ by different  branches from those  defining $\mathcal O^-(H_{ia})$: we consider the inverse branches of the dynamics defining $\tilde \Gamma^{iu}_a$.  This defines a sequence of preimages $(\tilde H_{ia}^{(k)})_{k\le 0}$ so that  $\tilde H_{ia}^{(k)}$ is close to 
$S_{ia}^{(k)}$ for every $k\le -1$.

We put $\mathcal O(\tilde H_{ia}):= \mathcal O^-(\tilde H_{ia})\cup \mathcal O^+(\tilde H_{ia})$. It is a discrete set with a unique accumulation point $P_a$. Also we notice that for every $i\not =j$ and $a\in U'_i\cap U'_j$:
\begin{itemize}
\item the $\delta$-neighborhood of $\mathcal O(\tilde H_{ja})$ is disjoint from $B(S^{(-1)}_{ia},\delta)$,
\item  $B(S^{(-1)}_{ia},\delta)$ contains 
$ \tilde H^{(-1)}_{ia}$ and is disjoint from the $\delta$-neighborhood of $\mathcal O(\tilde H_{ia})\setminus \{\tilde H^{(-1)}_{ia}\}$.
\end{itemize}

Then Proposition \ref{sinkcreation2} implies:

\begin{lemm}
For every $M\ge 0$, for every $i$ there exists a smooth perturbation $(f_{ia})_{a\in U_i}$ of  $(f_a)_{a\in U_i}$ so that:
\begin{itemize}
\item  $(f_{ia})_a$ has a sinks $A_{ia}$ of period $\ge M$ with orbit in $B(\mathcal O(H_{ia}),\delta)$, for every $a\in U_i'$. 
\item $f_{ia}(z)=f_a(z)$ for every $z\notin B(S^{(-1)}_{ia},\delta)$ and $a\in U_i$. 
\end{itemize}
\end{lemm}

Note that by the first item, for every $a\in U_i\cap U_j'$ with $i\not= j$, the orbit of $A_{ja}$ does not intersect $B(S^{(-1)}_{ia},\delta)$. Hence  for every perturbation $(f'_a)_a$ of $(f_a)_a$ so that:
\begin{itemize}
\item   $f_a'(z)= f_a(z)$ for every $z\notin B(S^{(-1)}_{ia},\delta)$ and every $i$ such that $a\in U_i$,
\item $f'_a= f_{ia}(z)$ for every $a\in U'_i$ and every $z\in B(S^{(-1)}_{ia},\delta)$,
\end{itemize}
The point $A_{ia}$ is still an attracting cycle of period $\ge M$ for $f'_a$, for every $a\in U'_i$. As $(U'_i)_i$ is a covering of $\I^k$, the family $(f_a')_a$ displays a sink of period $\ge M$ for every $a \in \I^k$.

We notice that the perturbation $(f_a')_a$ exists since the sets 
$( \{(a,z)\in U_i\times M: z\in B(S^{(-1)}_{ia},\delta)\})_i$ are disjoint.

\section{Replications of the source $S$ (Proof of Prop. \ref{propfonda3})}
\label{proofpropfonda3}
Let $(f_a)_a$ be a smooth family in $\mathcal U^\infty\subset \mathcal U^{d,\infty}$. Let $\mathcal F^{uu}_a(S_{a})$ be the strong unstable foliation associated to the projectively hyperbolic fixed source $S_a$ of $f_a$. 
 
\begin{prop}\label{propprefinal} For every $\epsilon>0$, for every $a_0\in \mathbb I^k$,  there exist $(k+1)3^k$-sources $(S_{i a_0})_i$ with disjoint periodic orbits so that:
\begin{enumerate}[$(1)$]
\item the source $S_{i a_0}$ is projectively hyperbolic and is $\epsilon$-close to $S_{a_0}$ and the $C^d$-jet  $J^d_{a_0} (S_{ia})_a$  of $(S_{ia})_a$ at $a=a_0$ is $\epsilon$-close to $J^d_{a_0}(S_a)_a$. 
\item the set $B$ is included in the basin of $S_{ia_0}$ and the leaves of the strong unstable foliation associated to $\mathcal F^{uu}(S_{i a_0})$ are $\epsilon$-$C^2$-close to those of $\mathcal F^{uu}(S_{i a_0})$, over $B$.
\end{enumerate}
\end{prop}

\begin{proof} For the sake of simplicity,  we prove this proposition for:
\[a_0=0.\]

Let $\arr Q\in \arr K$ be such that $S_0$ is in  $W^u_{loc}(\arr Q; f_{0})$. The preorbit $\arr Q$ being not necessarily periodic, it does not need to have Lyapunov exponents well defined. Let $(\lambda_n)_n$ and $(\sigma_n)_n$ be defined by:
\[\lambda_n = \|D_Qf^{-n}_0|E^s\|^{-1}<1 \quad \text{and}\quad   \sigma_n =\|D_Qf_0^{-n}|E^u\|^{-1}>1\; .\]
Let $\sigma_u$ and $\sigma_{uu}$ be the weak and strong unstable eigenvalues of $S_0$.

Let $\phi_a$ be the inverse branch of $f_{a}$ which contracts $B$ to $S_{a}$.  Then  $\phi_{0}^n$ contracts $B$ to a small neighborhood $B_n$ of $S_0$, with a contraction of the order of $\sigma_{u}^{-n}$.  
Let $(\psi^m_a)_{m\ge 0}$ be the inverse branches of $(f_a^m)_{m\ge 0}$ which define $W^u(\arr Q;f_a)$. For every $m$ large, $\psi^m_0$ is defined from a small neighborhood of $S_0$ with image in a small neighborhood of the zero coordinate of $\arr f^{-m}(\arr Q)$.
Its Lipschitz constant is of the order of $\lambda_{m}^{-1}$. The following lemma is obvious:
\begin{lemm}\label{choixnm}
For every $m$ large, for every $n\ge 0$ large enough, it holds:
\begin{enumerate}[$(a)$] 
\item $\lambda_{m}\cdot \sigma_{u}^{n}$ is large,
\item $\lambda_{m }\cdot \sigma_{u}^{n}$ is small compare to $\sigma_{m}\cdot \sigma_{uu}^{n}$. 
\end{enumerate}
\end{lemm}
In particular, for $n$ large enough, 
for such a choice of $n,m$, the map $\phi_0^n\circ \psi^m_0\circ \phi^n_0$ is well defined on $B$ and is very contracting. We put $B_{n,m} = \psi^m_0(B_n)= \psi^m_0\circ \phi^n_0(B)$ and 
$B_{n,m,n} = \phi^n_0(B_{n,m})= \phi^n_0\circ\psi^m_0\circ \phi^n_0(B)$ . 

By $(H_1)$,  $B_{n,m,n}$ is included in the small neighborhood $B_n\subset B$ of $S_0$, and so in $B$. Consequently $\phi^n_0\circ\psi^m_0\circ \phi^n_0$ has a fixed point $S_{i0}$ in $B_{n,m,n}$. As $B_{n}$ is a small neighborhood of $S_{0}$, the $f_0$-periodic source  $S_{i0}$ is close to $S_{0}$. We notice that $B$ is in the basin of $S_{i0}$. Let us show that the continuation $(S_{ia})_a$ of $S_{i0}$ satisfies $(1)-(2)$ for $n,m$ sufficiently large.  

\paragraph{(1)}  Let us come back to the notations introduced in \textsection \ref{notationJdM} on the space of $d$-jets $J^d_0 (\I^k ,M)$.  We remark that the map $J^d_{0} (\phi_a^n)_a$ is as contracting    as $\phi_0^n$ with a constant of the order of $\sigma_{u}^{-n}$, whereas $J^d_{0} (\psi_a^m)_a$ is Lipschitz with a constant of the order of $\lambda_{m}^{-1}$. Hence, by $(a)$, the composed map   $J^d_{0}  (\psi_a^n\circ \phi_a^m\circ \phi_a^n)_a$ is very contracting with a unique fixed point $J^d_{0} (S_{ia})_a$ close to $J^d_{0} (S_{a})_a$.

 By hyperbolicity and condition $(iv)$, every line $L$ close to the line field $E^{uu}(S_0)$ and pointed at a point in  $B_{n,m,n}$, is sent by  $Df^m$ to a line $L'$ even closer to the line field $E^{uu}(S_0)$ and pointed at a point in $B_{n,m}$. Then it is sent by $Df^m$ to a line  $L''$ close to $TW^u(Q; f_{0})$ (and pointed at a point in $B_n$). Finally, it is sent by $Df^n$ to a line very close to $E^{uu}(S_0)$ and pointed at a point in $B$.  Hence there is an invariant cone field, and so $S_{i0}$ is a projectively hyperbolic periodic source.  

\paragraph{(2)}  We proved that the strong unstable direction $E^{uu}_i(S_0)$ associated to $S_{i0}$ is $C^0$-close to $E^{uu}(S_0)$. 
Let us show that the leaves integrating $E^{uu}_i(S_0)$ are $C^2$-close to the leaves of $\mathcal F^{uu}(S_0)$.
 Let $\gamma$ be a curve in $B_{n,m,n}$ which is $C^2$-close to a plaque of $\mathcal F^{uu}(S_0)$. 
By projective hyperbolicity, the image by $f_0^n$ of $\gamma$ is a curve $C^2$-close to a plaque of $\mathcal F^{uu}(S_0)$ in $B_{n,m}$. Then by the inclination lemma and $(H_4)$, its image by  $f^m_0$ is $C^2$-close to a segment of $W^u_{loc}(Q;f_0)$ in $B_m$. Again by projective hyperbolicity, its image by $f^n_0$ is then $C^2$-close to a plaque of $\mathcal F^{uu}(S_0)$. This shows that a curve in $B_{n,m,m}$ which is $C^2$-close a plaque in $\mathcal F^{uu}(S_0)$ is sent by $f^{2n+m}_0$ to a curve which is even closer to a plaque of $\mathcal F^{uu}(S_0)$. This proves that the leaves $\mathcal F^{uu}_i(S_0)$ are $C^2$-close to those of $\mathcal F^{uu}(S_0)$.
\medskip 

Consequently for every $m$ large, for every $n$ large depending on $m$, there is a $2n+m$-periodic sources $(S_{ia})_a$ which satisfies $(1-2)$ at $a=0$.  By taking different values for $m$, we get $(k+1)3^k$ periodic sources $(S_{ia})_a$ with disjoint orbits as claimed in the Lemma.
\end{proof}

For every $a\in \I^k$, all the estimates involved in the algorithm given in the above lemma are bounded from above. Hence by compactness of $\I^k$, the integers $n, m$ are bounded by a constant $N$ independent of  $a\in \I^k$. 

This implies the following:
\begin{cor}\label{prefinal}
There exist $\eta>0$ and $N>0$ so that for every $a_0\in \I^k$, 
there exist  a finite set $\hat J(a_0)$  of symbols and  a family  $(S_{ia_0})_{i\in \hat J(a_0)}$ of periodic sources $S_{ia_0}$ so that:
\begin{enumerate}[$(a)$]
\item The cardinality of $\hat J(a_0)$ is  $(k+1) 3^k$ and for every $i\not = j\in \hat J(a_0)$, the sources $S_{ia_0}$ and $S_{ja_0}$ have disjoint orbits,
\item For every $i\in \hat J(a_0)$, the source $S_{ia_0}$ persists for every $a_1\in a_0+[-\eta,\eta]^k$, and its continuations $(S_{ia})_a$  satisfy $(1)$ \& $(2)$ at $a_1$.  
\item The period of $S_{ia_0}$  is bounded by $3N$ and the expansion of $S_{ia}$ is bounded from below by $1000$ for every $a\in a_0+[-\eta,\eta]^k$.
\end{enumerate}
\end{cor}

 For every $a_0\in \I^k$, for every $i\in \hat J(a_0)$, for every $a\in a_0+[-\eta,\eta]^k$, we define the finite set:
\[\mathcal O^+ (S_{ia}):= \{f^k_a(S_{ia}): \; k\ge 0\}.\]

By $(c)$,  there exists $\delta>0$, 
s.t. for all $a_0\in \I^k$ and $i\in \hat J(a_0)$, the map $f^{(3N)!}_{a_0}|B(S_{i a_0}, \delta)$ is expanding and so $S_{i a_0}$ is its unique fixed point. Note that $(3N)!$ is divided by the periods of all the sources $S_{ia_0}$ among $i\in   \hat J(a_0)$. 
 Hence, for all $a_0,a_1\in \I^k$, so that $a_0+
  [-\eta,\eta]^k\cap a_1+[-\eta,\eta]^k$ contains a parameter $a$, for all $i \in \hat J(a_0)$ and $j \in \hat J(a_1)$, it holds:  
\[\text{either }\mathcal O^+ (S_{ia})=  \mathcal O^+ (S_{ja})\quad \text{or}\quad d(\mathcal O^+ (S_{ia}), \mathcal O^+ (S_{ja}))>\delta\; .\]

For every $\eta'<\eta$, let $\Z_{\eta'}:=  \I^k \cap \eta' \Z^k$. 
\begin{lemm} 
For every $\eta'<\eta$, there exists a subset $\sqcup_{a_0\in \Z_{\eta'}} J(a_0)$ of the disjoint union $\sqcup_{a_0\in \Z_{\eta'}} \hat J(a_0)$ so that:
\begin{itemize}
\item  for all $a_0\in  \Z_{\eta'}$,  the set $J(a_0)$ has cardinality $(k+1)$, and for every $a\in a_0+[-\eta',\eta']^k$ and all $i\not=j\in J(a_0)$, the orbits $\mathcal O(S_{ia})$ and $\mathcal O(S_{ja})$ are $\delta$-distant. 
\item for all  $a_1\not= a_2 $ , for every  $i\in J(a_1)$ and $j\in J(a_2)$, for every $a\in (a_1+[-\eta',\eta']^k)\cap (a_2+[-\eta',\eta']^k)$, the orbits $\mathcal O(S_{ia})$ and $\mathcal O(S_{ja})$ are $\delta$-distant. 
\end{itemize}
  \end{lemm}
  \begin{proof}
Let us index $\Z_{\eta'} =:\{a_i: 1\le i\le q\}$.  Let $J(a_1)\subset \hat J(a_1)$ be any subset of cardinality $k+1$. By Corollary \ref{prefinal} (a), the orbits of $S_{ia}$ and $S_{ja}$ are disjoint for every $i\not =j\in J(a_1)$ and $a\in  (a_1+[-\eta',\eta']^k)$. 

Let $2\le q'\le q$ and assume by induction $J(a_i)$ constructed for every $i< q'$.  We notice that the cardinality of
$\Z_{\eta'}(q'):= \{ a_i \in a_{q'} +[-\eta',\eta']^k: i<q'\}\cap \Z_{\eta'}$  is at most $3^k-1$.
By induction, the cardinality of $\sqcup_{ a_i \in \Z_{\eta'}(q')} J(a_i)$ is at most $ (k+1)  (3^k-1)$ .
 Hence we have to remove at most $ (k+1)  (3^k-1)$ periodic sources of $\hat J(a_{q'})$ so that the remaining sources have continuations with disjoint orbit to those indexed by  $\sqcup_{ a_i \in \Z_{\eta'} (q')} J(a_i) $.  We chose any set $J(a_{q'})$ of cardinality $k+1$ in the remaining set formed by at least $(k+1) 3^k-  (k+1) (3^k-1)= k+1$ different sources.
 \end{proof}
  
We recall that there exists a continuous family of local unstable manifolds  $(W^u_{loc}(\arr z; f))_{\arr z\in \arr K}$ so that by $(H_3)$ and Proposition \ref{propprefinal}.2, for every $i\in  \sqcup_{a_0\in \Z_{\eta'}} J(a_0)$, there exists $\arr z_i\in \arr K$ satisfying:
\begin{itemize}
\item $W^u_{loc}(\arr z_i; f_{a_0})$ contains $S_{i,a_0}$,
\item there exists a $C^d$-family  of points $(Q_a)_a\in (W^u_{loc}(\arr z_i; f_{a}))_a$ and a continuous function $\epsilon_i$ equal to zero at $0$ s.t:
\[d(Q_a, S_{i,a})\le  \epsilon_i( \|a-a_0\|)\cdot \|a-a_0\|^d.\] 
\end{itemize}

We recall that the family  $(W^u_{loc}(\arr z; f_a)_{a\in \I^k})_{\arr z\in \arr K}$ is continuous in the $C^{\infty}$ topology. 

Hence there exists a family of smooth charts $(\phi_{ia})_{a\in \I^k} $ from a neighborhood of $W^u_{loc}(\arr z_i; f_{a})$ onto an open subset of $\R^2$, which send $W^u_{loc}(\arr z_i; f_{a})$ to the constant segment $[-1,1]\times \{0\}$ for every $a\in \I^k$ and which have bounded $C^{r,r}$-norm independently of $a_0\in \Z_{\eta'}$ and $i\in J(a_0)$ for every $r\ge 0$. 

We remark that $S_{ia}$ belongs to the domain of this chart for $a$ sufficiently close to $a_0$. Let $(x_{i}(a),y_{i}(a)):=\phi_a(S_{ia})$ and remark that:
$$\partial_a^s y_{i}(a_0)=0\quad \forall s\le d\; .$$  

Moreover, by Corollary \ref{prefinal}.$(c)$,
for every $a_0\in \Z_{\eta'}$ and $i\in J(a_0)$, 
the period $p$ of $S_{ia}$ is at most $3N$ and the expansion of $D_{S_{ia}}f^p_a$ is at least $1000$. Thus the following derivative :
\[a\in a_0+ [-\eta,\eta]^k\mapsto \partial_a  S_{ia} = - (D_{z} f_a^p(S_{ia}) )^{-1}(\partial_a f_a^p)(S_{ia})\]
displays a $C^{d}$-norm bounded independently of $\eta'$, $i$,  $a_0$, and $a$. In particular,  there exists $C_{d+1}>0$ so that for every $a\le \eta$, it holds:
$$|\partial_a^{d+1} y_{i}(a_0+a)|\le C_{d+1}\; , \text{where }  (x_{i}(a),y_{i}(a)):=\phi_a(S_{ia}) \quad \text{and}\quad  \partial_a^s y_{i}(a_0)=0, \; \forall s\le d$$


A crucial point is that $C_{d+1}$ and $\delta>0$ depend neither  $\eta'$ nor on $a_0,i$.  
 
Hence for $\eta'\in (0,\eta)$ sufficiently small, we can $C^{d,\infty}$-perturb $(f_a)_a$ to a smooth family $(f'_a)_a$ so that  $S_{ia}$ persists as $S_{ia}':=  \phi_a^{-1}(x_{ia},0)$ for every $a\in a_0+[-2\eta'/3,2\eta'/3]$, the perturbation being supported by $(a,z)\in (a_0+[-\eta',\eta']^k)\times B(S_{ia},\delta/2)$. 

Note that the perturbation is $C^{d,\infty}$-small when $\eta'$ is small. This proves:
\begin{prop}
For every $\eta'<\eta$ small, there exists a $C^{d,\infty}$-perturbation $(f'_a)_a\in \mathcal U^\infty$ of $(f_a)_a$ and families of sources $(S'_{ia_0})_{ 
a_0\in \Z_{\eta'}, i\in J(a_0)
}$ so that 
\begin{itemize}
\item  for every $a_0\in  \Z_{\eta'}$,  the cardinal of the set $J(a_0)$ is  $(k+1)$,
\item for all $a_1,a_2\in \Z_{\eta'}$, for all 
$i \not = j \in J(a_1) \sqcup J(a_2)$, for every $a\in (a_1+[-\eta',\eta']^k)\cap (a_2+[-\eta',\eta']^k)$, the orbits $\mathcal O(S'_{ia})$ and $\mathcal O(S'_{ja})$ are $\delta/2$-distant. 
\item for every  $a_0\in  \Z_{\eta'}$, for all $i\in J(a_0)$, there exists $\arr z_i\in \arr K$ so that $S'_{ia}$ belongs to $W^u_{loc}(\arr z; f_a)$ for every $a\in a_0+[-2\eta'/3,2\eta'/3]^k$.
\end{itemize}
  \end{prop}
  
By Taking $\eta'$ such that $1/\eta'\in \N$, it holds that $\cup_{a_0\in \Z_{\eta'}} (a_0+[-2\eta'/3, 2\eta'/3]^k)\supset \I^k$.  
  
Note that by $(H_2)$ a  segment $W^s_{loc}(P, f_a')$
 has a robust tangency with $\mathcal F^{uu}(S_a)$ for every $a\in \I^k$. Also if $a\in a_0+(-\eta,\eta)^k$
  and $a_0\in  \Z_{\eta'}$, and $i\in J(a_0)$, the leaves of 
  the foliation $\mathcal F^{uu}(S_{ia})$ are $C^2$-close to $\mathcal F^{uu}(S_a)$. 
We recall that $\mathcal F^{uu}(S_{ia})$ is actually a $C^\infty$-foliation depending smoothly on the parameter by Prop. \ref{Fuusmooth}. 

We recall that $\mathcal U$ which is defined by $(H_0-H_1-H_2)$ 
is a $C^1$-open set which allows the robust tangency between $\mathcal F^{uu}(S)$ and $W^s_{loc} (P)$ to be not quadratic.

However, by performing a small perturbation of the dynamics in the $\delta/2$-neighborhood of $S_{ia}'$, we can keep $S_{ia}'$ and  $W^s_{loc}(P, f_a')$  at the same places and  make any small smooth perturbation for  $(\mathcal F^{uu}(S_{ia}))_a$.
 By Thom's transversality theorem, for a typical perturbation of $(\mathcal F^{uu}(S_{ia}))_a$, the set of parameters for which 
$W^s_{loc}(P, f_a')$ does not display a quadratic tangency with $\mathcal F^{uu}(S_{ia})$ is a (possibly empty) submanifold of codimension $1$ in $(-2\eta/3,2\eta/3)+a_i$. Let $U_i$ be the  complement of this manifold in $(-2\eta/3,2\eta/3)+a_i$. 

 Moreover these $k+1$ one-co-dimensional manifolds are multi-transverse for $i\in  J(a_0)$. 
As the parameter space $\I^k$ has dimension $k$, for a typical perturbation of $((\mathcal F^{uu}(S_{ia}))_a)_{i\in J(a_0)}$,
the union $\cup_{i\in J(a_0)}U_i$ contains $a_0+(-2\eta'/3,2\eta'/3)$.

Thus a $C^{d,\infty}$-perturbation $(f''_a)_a\in \mathcal U^\infty$ of $(f_a)_a$ satisfies Proposition \ref{propfonda3}   with sources $(S_{ia})_{a\in  U_i}$ among  $i\in \sqcup_{a_0\in \Z_{\eta'} }J(a_0)$. 
\section{Proof of Corollary \ref{theo2}}
\label{diffcase}
\begin{proof}
Let $X\in \{A,PS\}$ and $1\le d\le r<\infty$. 
Let us come back to Example \ref{examfonda}. Let $N$ be a small neighborhood of $(D\sqcup I_S\sqcup I_P\sqcup I_{P'})\times [-3,3]$. 

 We recall that for every $M\ge 0$, we have constructed a $\mathcal U^{d,r}_X$-dense set $D$ of $C^{d,\infty}$-families $(f_a)_a$ so that for every $a\in \I^k$, the map $f_a$ displays a sink of period at least $M$ with orbit in $N$. 

Let $\hat \R$ be the one point compactification of $\R$, and let $\mathbb A= \hat \R\times [-4,4]$. Let $I$ be a compact neighborhood of the infinity, so that $I\times [-4,4]$ does not intersect $N$.

As the map $Q\colon \sqcup_{\delta \in \Delta} I_\delta\sqcup I_S\sqcup I_P\to [-1,1]$ is orientation preserving and $f_a| I_{P'\times [-3,3]}$ as well, it is possible to extend each smooth family $(f_a)_a\in D$  to a family of local diffeomorphisms of $\mathbb A$ of degree $Card \, \Delta+3$ so that:
\begin{enumerate}[$(i)$]
\item for every $(f_a)_a\in D$, for every $a\in \I^k$, the map has a sinks of period $\ge M$ outside of $I\times [-4,4]$.
\item for every $\theta\in \hat \R$, the map $f|\{\theta\}\times [-4,4]$ is injective.
\end{enumerate}
Let $ \tilde {\mathbb A} :=  \mathbb A \times (-1,1)^{n-2}$. Let $M$ be a manifold of dimension $n$. 
By section 5, \textsection [inv. in dim. $\ge 4$ and $=3$] \cite{BE15}, we can find a small function $\rho \in C^\infty(\mathbb A,(-1,1)^{n-2})$   so that for $\lambda>0$ small compare to $\|\rho\|_{C^0}$,   for every $(f_a)_a\in \mathcal U^{d,r}_X$, the following is a restriction of a $C^\infty$-family of  diffeomorphisms of $M$:
\[\tilde f_a \colon (z,h)\in \tilde {\mathbb A}\mapsto (f(z), \lambda h + \rho(z))\in \tilde{ \mathbb A} \]
We notice that if $\pi$ denotes the projection $\pi(z,h)=z$, then it holds $\pi \circ \tilde f_a= f_a \circ \pi$.
In fact, $\pi$ is the holonomy along strong local stable manifolds of the form $W^{ss}_{loc} ((z,h); \tilde f_a)= \{z\}\times (-1,1)^{d-2}$. 

We recall that $\mathcal U^{d,r}_X$ is the intersection of $\mathcal U^{d,d}_A$ with  $C^{d,r}_X(\I^k, \mathbb A, \mathbb A)$ where $\mathcal U^{d,d}_A$ is the set of $C^{d,r}_A$-families which satisfy $(H_0\cdots H_4)$, \textsection 2.4. 

Let us fix an infinitely smooth family $(f_a)_a\in \mathcal U^{d,d}_A\cap C^{\infty}(\I^k\times \mathbb A,\mathbb A)$, and let $\tilde {\mathcal U}^{d,d}_A$  be a small $C^{d,d}_A$-neighborhood of $(\tilde f_a)_a$.  Put 
$\tilde {\mathcal U}^{d,r}_X = 
\tilde {\mathcal U}^{d,d}_A\cap C^{r}(\I^k\times \tilde{ \mathbb A} , \tilde{ \mathbb A} )$ , given a topology $X\in \{A,PS\}$ and $r\ge d$. 


\begin{lemm}[Lemma 5.1 \cite{BE15}]\label{5.1} For every $\infty >r\ge d\ge 1$, for $\tilde {\mathcal U}^{d,d}_A$  small enough, for every  $\lambda>0$ small enough, for every $(\tilde f_a')_a \in \tilde {\mathcal U}^{d,d}_A\cap C^{d+r+1}(\I^k\times \tilde{ \mathbb A}, \tilde{ \mathbb A})$,  the continuities $(W^{ss}_{loc} (z; \tilde f'_a))_{z\in \mathbb A}$ of the  strong unstable manifolds form a fibration which is of class $C^{r+d}$
 for every $a\in \I^k$. 
Moreover, the family $(\cup_{a\in \I^k} \{a\}\times W^{ss}_{loc} (z; \tilde f'_a))_{z\in \mathbb A}$ is of class $C^{d+r}$
 and  its family of tangent space is $C^{d-1,d-1}_A$-close to the family $(TW^{ss}_{loc} (z; \tilde f_a))_{z\in \mathbb A}$.
\end{lemm}

Hence given $(\tilde f_a')_a \in \tilde {\mathcal U}^{d,d}_A\cap C^{d+r+1}(\I^k\times \tilde{ \mathbb A},\tilde{ \mathbb A})$, the holonomy of along the strong stable foliation to the transverse section $\{h=0\}$ defines a $C^{d+r,d+r}_A$-family $(\pi'_a)_a$ of projections $\pi_a'\colon \tilde{ \mathbb A}\to \mathbb A$ which is $C^{d-1,d-1}_A$-close to $(\pi)_{a\in \I^k}$. 
Let $(f_a')_a\in C^{d+r,d+r}_A(\I^k,\mathbb A,\mathbb A)$ be defined by:
\[ \pi_a'\circ \tilde f'_a= f_a'\circ \pi_a'\; .\]
Unfortunately, $(f_a')$ is in general only $C^{d-1,d-1}_A$-close to $(f_a)_a$. Nevertheless we are going to show that:
\begin{fact}\label{dernierfact}
The $C^{d+r,d+r}_A$ family  $(f_a')_a$ satisfies $(H_0, H_1, H_2,H_3, H_4)$.
\end{fact}
Consequently, by the main theorem,  for every $N\ge 0$, there exists a $C^{d+r,d+r}_A$-perturbation $(f_a'')_a$ of $(f'_a)_a$ (which is also a $C^{d,r}_X$-perturbation) so that the map  $f''_a$ displays a sink of period $\ge N$ for every $a\in \I^k$. 

 
Let $\tilde f''_a$ be the map which sends $(z,h)\in \tilde {\mathbb A }$ to 
the point in $\pi'^{-1}_a(\{f''_a(z)\})$ with $(n-2)$-last coordinates equal to those of $\tilde f'(z,h)$. We observe that:
\[\pi'_a \circ \tilde f''_a = f''_a \circ \pi'_a \; .\]
Furthermore, by smoothness of the strong stable foliation, 
$\tilde f''_a$ is a $C^{d+r,r+d}_A$-perturbation $(\tilde f'_a)_a$.  In particular $(\tilde f''_a)_a$ is a $C^{d,r}_X$-perturbation of $(\tilde f'_a)_a$.

Also, since the fibers are contracted,  for every $a\in \I^k$,
 $\tilde f''_a$ has a sinks of period $\ge N$.  This proves the existence of a $C^{d,r}_X$-dense set of family of diffeomorphisms which display  a sink of period $\ge M$ for every parameter $a\in \I^k$. Note that this set is necessarily open. The intersection of these open and dense sets among $N\ge 0$ is the residual set $\mathcal R$: a family therein displays a sink of arbitrarily large period and so infinitely many sinks, at every parameter $a\in \I^k$.   
 \end{proof}
\begin{proof}[Proof of Fact \ref{dernierfact}]

We recall that $(f_a)_a$ satisfies $(H_0\cdots H_4)$ with a family of projectively hyperbolic source $(S_a)_a$, an area contracting fixed point $(P_a)_a$, a $C^d$-parablender $(K_a)_a$ and a family of local unstable manifold $(W^u_{loc} (\arr z; f_a))_{\arr z\in \arr K\; a\in \I^k}$. These sets are cannonically embedded in $\{h=0\}$ to hyperbolic set for the the product dynamics of $(f_a)_a$ with $0$. Hence for $\lambda>0$ and $\rho$ small, they persists for $(\tilde f_a)_a$ to hyperbolic set. Let $(\tilde S_a)_a$, $(\tilde P_a)$, $(\tilde K_a)_a$ and $(W^u(\arr z; \tilde f_a))_{\arr z\; a}$ be the hyperbolic continuations of them for $(\tilde f_a)_a$.  Let $(\tilde S'_a)_a$, $(\tilde P'_a)$, $(\tilde K'_a)_a$  and $(W^u(\arr z; \tilde f'_a))_{\arr z\; a}$ be the hyperbolic continuation of them for the $C^{d+r+1}$-family $(\tilde f'_a)_a$. 
Let $a\in \I^k$.  Note that the local strong stable manifold of $\tilde S'_a$ and $\tilde P'_a$ intersect $\{ h=0\}$ at respectively a projectively hyperbolic source $S_a'$ and an area contracting saddle point $P'_a$  for $f'_a$. 

Let us prove $(H_1)$ and $(H_2)$. 
The dynamics of $\tilde f_a'$ restricted to a local unstable manifold of $W^u_{loc} (\tilde S_a'; \tilde  f_a')$ is $C^d$-close to $\tilde f_a| W^u_{loc} (\tilde S_a; \tilde f_a)$. 
By $(H_1)$, the latter can be chosen so that it projects diffeomorphically onto $B$ by $\pi$.  Hence $W^u_{loc} (\tilde S_a'; \tilde  f_a')$ project diffeomorphically by $\pi_a'$  onto a domain $B'$ which is $C^0$-close to $B$: a basin of $S_a'$.
By hyperbolic continuation, the image $K'_a=\pi_a'(\tilde K'_a)$ is close to $K_a$ (for the Hausdorff topology on compact subset). Hence, it is included $B'$ by $(H_1)$ for $f_a$. This shows the first part of $(H_1)$ for $f'_a$, for every $a\in \I^k$. 

Since $\pi_a'| W^u_{loc} (\tilde S_a'; \tilde  f_a')$ is a diffeomorphism, it suffices to verify the second part of $(H_1)$ and the first part of $(H_2)$ at  $W^u_{loc} (\tilde S_a'; \tilde  f_a')$. 

For this end, observe that $W^u_{loc} (\tilde S_a'; \tilde  f_a')$ is foliated by strong unstable manifolds $\tilde{\mathcal F}^{uu}$, and this foliation is $C^d$-close to the one of $W^u_{loc} (\tilde S_a; \tilde  f_a)$.  
By $(H_2)$ for $f_a$, the latter foliation displays a robust tangency with $W^s_{loc}(P_a; f_a)\times (-1,1)^k$, which is $C^d$-close to a local stable manifold $W^s_{loc}(\tilde P_a; \tilde f'_a)$  of $\tilde P'_a$. Hence the strong unstable foliation of   $W^u_{loc} (\tilde S_a'; \tilde  f_a')$ displays a robust tangency with $W^s_{loc}(\tilde P_a; \tilde f'_a)$. 
Looking at the image by the diffeomorphism $\pi_a'| W^u_{loc} (\tilde S_a'; \tilde  f_a')$, it comes that $W^s(P_a'; f'_a)$ displays a robust tangency with the strong unstable foliation of $S'_a$ : the first part of condition $(H_2)$ for $f_a'$. 

Likewise, for every $z\in K_a$, a small local stable manifold 
$W^s_{loc} (z; f_a)$ for $f_a$ defines a local stable manifold $W^s_{loc} ( z; f_a)\times (-1,1)^{n-2}$ for $\tilde f_a$. 
The latter is transverse to $\tilde{\mathcal F}^{uu}$ by $(H_1)$. 
Hence  by hyperbolic continuation, the final part of $(H_1)$ is also satisfied for $f_a'$.  

Furthermore, by $(H_2)$ for $f_a$, the local stable manifold $W^u(\tilde P_a ;\tilde f_a)$ displays a transverse intersection with $W^s_{loc} ( z; f_a)\times (-1,1)^{n-2}$ for a certain $z\in K_a$. By hyperbolic continuation, 
$W^u(\tilde P_a ;\tilde f_a)$ displays a transverse intersection with $W^s_{loc} ( \tilde z; \tilde f'_a)$ for $\tilde z$ in $\tilde K_a'$. 
This transverse intersection projects by $\pi'_a$ to a transverse intersection between $W^u(P'_a ; f'_a)$ and $W^s_{loc} ( \pi_a'(\tilde z);  f'_a)$. This shows the final part of $(H_2)$ for $f_a'$.  

 
Therefore $f'_a$ satisfies $(H_1)$ and $(H_2)$ for every $a\in \I^k$. 

\medskip
Let us show $(H_4)$.  We recall that a local strong stable manifold of $\tilde S_a$ is a vertical segment. It persists to local strong stable manifold $W^{ss}_{loc} (\tilde S_a'; \tilde f'_a)$ for $\tilde S'_a$ which is close to a vertical segment.

Let $E^c_a$ (resp. $E'^c_a$) be the sum of the strong stable direction and the weak unstable direction of $\tilde S_a'$ (resp. $\tilde S_a'$). The one-codimensional plane  $E'^c_a$ extends to a unique plane field $E'^c_a$ over $W^{ss}(\tilde S'_a; \tilde f'_a)$ which is at most weakly expanded. This is easily shown by using a cone field argument. The latter shows moreover that $E'^c_a$ is $C^0$-close to the constant plane field equal to $E^c_a$.
A vector in $D\pi'_a( E'^c_a)$ is at most weakly expanded by 
$D_{S'_a} f_a'^n$ since $Df_a'^n\circ D\pi_a' = D\pi_a' \circ D\tilde f_a^n$. Hence $D\pi'_a( E'^c_a)$ is equal to the weak unstable direction of $S'_a$. 

On the other hand, let us consider  a local unstable manifold $W^u_{loc} (\arr z; f_a)$ of the family satisfying  $(H_3-H_4)$ for $(f_a)_a$. It persists to a local unstable manifold $W^u_{loc} (\arr z; \tilde f'_a)$ which is $C^1$-close to   $W^u_{loc} (\arr z; f_a)\times \{0\}$. The latter cannot be tangent to $E^c_a$ by $(H_4)$ for $(f_a)_a$. Hence $W^u_{loc} (\arr z; \tilde f'_a)$ cannot be tangent to $E'^c_a$. 
Note that $W^u_{loc} (\arr z; \tilde f'_a)$ is sent by $\pi_a'$ to $W^u_{loc} (\arr z; f'_a)$. Thus $W^u_{loc} (\arr z; f'_a)$ cannot be tangent to the weak unstable direction of $S_a'$. This achieves the proof of $(H_4)$.

\medskip
Let us show $(H_3)$.  Let us call a vertical $\text{codim }2$-$C^d$-submanifold any product of a $C^d$-curve in $\mathbb I^k\times \mathbb A$ with $(-1,1)^{n-2}$.
By strong contraction, we notice that there is an open set $\mathcal V$ of  $\text{codim }2$-$C^d$-submanifold $C^d$-close to be vertical, which is left invariant by $(a,z)\mapsto (a,\tilde f_a(z))$ and $(a,z)\mapsto (a,\tilde f'_a(z))$.

As $\cup_{a} \{a,R_a\} \times  (-1,1)^d=\cup_{a} \{a\}\times W^{ss}_{loc} (\tilde R_a; \tilde f_a)\}$ is in $\mathcal V$, the hyperbolic continuation $\cup_{a} \{a\}\times W^{ss}_{loc} (\tilde R'_a; \tilde f'_a)\}$ is in $\mathcal V$ as well.  Hence its intersection $(S'_a)_a$  with $\{h=0\}$ is $C^d$-close to $(S_a)_a$. 

To accomplish the proof of $(H_3)$, it suffices to show that $(K'_a)_a$ is a $C^d$-parablender with $\hat O$ as covered domain (see Example \ref{expparablender} for the explicit definition of $\hat O$). 

Let  $a_0\in \I^k$ and let $(z_a)_a$ be a $C^d$-curve with 
$J^d_{a_0}(z_a)_a\in \hat O$. 
Let $W^{ss}(\tilde z_a; \tilde f_a')$ be the strong stable manifold of $\tilde z_a= (z_a,0)\in \tilde {\mathbb A}$. We notice that 
$\cup_{a\in \I^k} \{a\}\times W^{ss}(\tilde z_a; \tilde f_a')$ is in $\mathcal V$ and so $(W^{ss}(\tilde z_a; \tilde f_a'))_a$
is  $C^{d,d}_A$-close to the family $(\{z_a\}\times (-1,1)^d)_a$. Hence, by the covering property of Example \ref{expparablender}, there exists $\delta_{-1}\in \Delta$ so that the intersection point $\{(z_a(h),h)\}=W^{ss}_{loc}(\tilde z_a; f_a)\cap \mathbb A\times \{h\}$ satisfies that $J^d_{a_0} (z_a(h))_a$ is close to $\hat O_{\delta_{-1}}$, for every $h\in (-1,1)^{d-2}$.
  Let $W^s_{loc}(\tilde z^{-1}_a; \tilde f'_a)$ be the component of $\tilde f_a^{-1}(W^s_{loc}(\tilde z_a;\tilde  f'_a))\cap \tilde {\mathbb A}
  $ associated to the symbol $\delta_{-1}$. By the covering property, the intersection point $z_a^{-1}$ of $W^s_{loc}(\tilde z^{-1}_a; \tilde f'_a)$ with $\{h=0\}$ displays a $C^d$-jet at $a=a_0$ in $\hat O$. By definition, $z_a^{-1}$ is the preimage by $f'_a$ of $z_a$ associated to $\delta_{-1}$.
  
Again, we notice that $W^{ss}_{loc}(\tilde z^{-1}_a; \tilde f'_a)$   is $C^{d,d}_A$-close to the family $(\{z^{-1}_a\}\times (-1,1)^d)_a$, and so there exists $\delta^{-2}\in \Delta$ so that $W^{ss}_{loc}(\tilde z^{-1}_a; \tilde f'_a)$ is made of points $z_a^{-2}$ s.t. $J^d_{a_0} (z_a)_a$ is close to be in $\hat O_{\delta_{-2}}$. 
  And so on, we define likewise a sequence of preimages $((z^{-i}_a)_{i\le -1})_a$ of $(z_a)_a$ by $(f'_a)_a$ associated to 
a preorbit of symbols $\underline \delta\in \Delta^{\Z^-}$ ,
and such that $J^d_{a_0}(z^{-i}_a)_a$ is in $\hat O$ for every $i$.  Let $W^u_{loc}(\underline \delta ; f'_a)$
  be the local unstable manifolds  associated to the preorbit $\underline \delta$. 
  As for Example \ref{expparablender} this implies the existence of a $C^d$-family of points $(Q_a)_a\in (W^u_{loc} (\underline \delta ; f'_a))_a$ satisfying $J^d_{a_0}(Q_a)_a = J^d_{a_0} (z_a)_a$.   
  This proves that $\hat O$ is a covered domain of the $C^d$-parablender $(K'_a)_a$. 
\end{proof}

\bibliographystyle{alpha}
\bibliography{references}

\end{document}